\documentclass{article}
\usepackage{arxiv}
\usepackage[utf8]{inputenc} 
\usepackage[T1]{fontenc}    
\usepackage{hyperref}       
\usepackage{url}            
\usepackage{booktabs}       
\usepackage{amsfonts}       
\usepackage{nicefrac}       
\usepackage{microtype}      
\usepackage{lipsum}
\usepackage{graphicx}
\usepackage{amsmath}
\usepackage{amssymb}
\usepackage{amsthm}
\usepackage[dvipsnames]{xcolor}

\theoremstyle{plain}
\newtheorem{mainthm}{Theorem}

\newtheorem{maincor}[mainthm]{Corollary}
\newtheorem{theorem}{Theorem}[section]
\newtheorem{lemma}[theorem]{Lemma}
\newtheorem{proposition}[theorem]{Proposition}

\theoremstyle{definition}

\newtheorem{remark}[theorem]{Remark}
\newtheorem{example}[theorem]{Example}

\newcommand{\Dcf}{D^c\!f}
\newcommand{\mf}{\mathbf{f}}
\newcommand{\mh}{\mathbf{h}}

\newcommand{\cF}{{\mathcal F}}
\newcommand{\cG}{{\mathcal G}}
\newcommand{\cM}{{\mathcal M}}
\newcommand{\cR}{{\mathcal R}}

\newcommand{\cU}{{\mathcal U}}
\newcommand{\cV}{{\mathcal V}}
\newcommand{\cW}{{\mathcal W}}
\newcommand{\CC}{{\mathbb C}}
\newcommand{\DD}{{\mathbb D}}
\newcommand{\HH}{{\mathbb H}}
\newcommand{\NN}{{\mathbb N}}
\newcommand{\PP}{{\mathbb P}}
\newcommand{\RR}{{\mathbb R}}
\renewcommand{\SS}{{\mathbb S}}
\newcommand{\TT}{{\mathbb T}}
\newcommand{\ZZ}{{\mathbb Z}}
\newcommand{\dist}{\operatorname{dist}}
\newcommand{\trace}{\operatorname{trace}}
\newcommand{\SL}{\operatorname{SL}}
\newcommand{\PSL}{\operatorname{PSL}}
\newcommand{\Id}{\operatorname{Id}}
\newcommand{\id}{\operatorname{id}}
\newcommand{\Stab}{\operatorname{Stab}}

\newcommand{\volume}{\operatorname{vol}}

\title{Hyperbolicity and rigidity\\ for fibred partially hyperbolic systems
\thanks{Project supported by Fondation Louis D. -- Institut de France (project coordinated by M. Viana).}}

\author{
Sankhadip Chakraborty\\
\thanks{Supported by a scholarship from CNPq - National Research Council of Brazil.}\\
Instituto de Matemática Pura e Aplicada\\
Estrada Dona Castorina 110, 22460-320 Rio de Janeiro, Brazil \\
\texttt{sankha01@impa.br} \\

   \And
Marcelo Viana\\
\thanks{Partially supported by CNPq and FAPERJ - State Research Agency of Rio de Janeiro.}
Instituto de Matemática Pura e Aplicada\\
Estrada Dona Castorina 110, 22460-320 Rio de Janeiro, Brazil \\
\texttt{viana@impa.br} \\
}

\begin{document}
\maketitle
\begin{abstract}
Every volume-preserving centre-bunched fibred partially hyperbolic system with 2-dimensional centre
either
(1) has two distinct centre Lyapunov exponents,
or
(2) exhibits an invariant continuous line field (or pair of line fields) tangent to the centre leaves,
or
(3) admits a continuous conformal structure on the centre leaves invariant under both the dynamics and
the stable and unstable holonomies.
The last two alternatives carry strong restrictions on the topology of the centre leaves:
(2) can only occur on tori, and for (3) the centre leaves must be either tori or spheres.
Moreover, under some additional conditions, such maps are rigid, in the sense that they are
topologically conjugate to specific algebraic models.
When the system is symplectic (1) implies that the centre Lyapunov exponents are non-zero,
and thus the system is (non-uniformly) hyperbolic.
\end{abstract}


\tableofcontents

\section{Introduction}\label{s.introduction}

In this paper we investigate the behaviour of the Lyapunov exponents of volume-preserving and symplectic
diffeomorphisms under small modifications of the dynamical system.
We are especially concerned with the following pair of fundamental questions:
\emph{How often can one perturb the diffeomorphism in order to make it (non-uniformly) hyperbolic?}
\emph{Which obstructions are there to the existence of such perturbations?}

In a nutshell, we conclude that non-uniformly hyperbolic systems are prominent
in the contexts we deal with and,
indeed, diffeomorphisms that are not approximated by non-uniformly hyperbolic
ones present very rigid features.
Our arguments benefit from a combination of methods, both classical and new, that have much broadened
this field of research in the last couple of decades or so.

The concept of Lyapunov exponents originated from the stability theory of differential equations developed
by A.M. Lyapunov~\cite{Lya92} around the end of the 19th century.
Consider a differential equation of the form
\begin{equation}\label{eq.quasi-linear}
x'= L(t)x + R(t,x),
\end{equation}
where $L(t):\RR^d\to\RR^d$ is a linear map and $R(t,x)$ is a non-linear perturbation of order greater than $1$.
Fix any $t_0$ and for each $v \neq 0$ denote by $x_v(\cdot)$ the solution of the \emph{linearised equation}
$x'=L(t)x$ with initial condition $x_v(t_0)=v$. The stability theorem of Lyapunov asserts that if the
\emph{Lyapunov exponent function}
\begin{equation}\label{eq.limsup}
\lambda(v) = \limsup_{t\to\infty} \frac 1t \log \|x_v(t)\|
\end{equation}
is negative for every $v \neq 0$ then, under an additional condition called \emph{Lyapunov regularity},
the solution $x(t) \equiv 0$ is exponentially stable for the equation \eqref{eq.quasi-linear}. 
See Barreira, Pesin~\cite{BaP02} for a detailed presentation of this topic.

The flow of a differentiable vector field may always may be written in the form \eqref{eq.quasi-linear}
around any given trajectory. Furstenberg, Kesten~\cite{FK60} proved that the limit in
\eqref{eq.limsup} exists (for all $v \neq 0$) for almost every trajectory,
relative to any measure \footnote{In this paper all measures are finite Borel measures.}
invariant under the flow.
A few years later, Oseledets~\cite{Ose68} proved that Lyapunov regularity also holds for
almost every trajectory, again with respect to any invariant measure.
Corresponding facts for discrete-time dynamical systems, such as diffeomorphisms on manifolds, 
follow easily. These results brought the subject of Lyapunov exponents firmly to the realm of
ergodic theory, where it has prospered since.

The next major step, initiated by Pesin~\cite{Pes76,Pes77}, was to develop the non-linear theory, 
namely, to establish that \emph{in the absence of vanishing Lyapunov exponents} the dynamical
system must exhibit special geometric features, including the existence of stable and unstable
sets that are smooth embedded disks at almost every point.
See also Ruelle~\cite{Rue81}, Fathi, Herman, Yoccoz~\cite{FHY83}, and Pugh, Shub~\cite{PSh89}.
Such a geometric structure is at the basis of several deep
results about \emph{non-uniformly hyperbolic systems}, that is,
whose Lyapunov exponents are non-zero almost everywhere,
by Pesin~\cite{Pes77}, Katok~\cite{Kat80c}, Ledrappier~\cite{Led84a}, Ledrappier, Young~\cite{LeY85a},
Barreira, Pesin, Schmeling~\cite{BPS99}, Shub, Wilkinson~\cite{SW00} and others.

All this points at the following central question:
\emph{Is every dynamical system approximated by a non-uniformly hyperbolic one?}
In other words, can one always get rid of zero Lyapunov exponents by slightly perturbing
the dynamical system?

Work of Herman (see Yoccoz~\cite[Section~4]{Yoc92}) implies that the answer may fail to be positive
in the context of volume-preserving diffeomorphisms.
Also, results of Mañé, Bochi~\cite{Man84,Boc02} show that vanishing Lyapunov exponents are actually
quite common
among $C^1$ area-preserving surface diffeomorphisms.
The latter was extended by Bochi, Viana~\cite{BcV05,Boc09} to volume-preserving and symplectic
diffeomorphisms in any dimension.

On the other hand, general perturbative techniques have been developed for removing vanishing
Lyapunov exponents.
See, in particular, Herman~\cite{Her83b}, Shub, Wilkinson~\cite{SW00}, Avila, Bochi~\cite{AvB02},
Dolgopyat, Pesin~\cite{DoP02}, Baraviera, Bonatti~\cite{BB03}, Bochi, Fayad, Pujals~\cite{BFP06},
and Marín~\cite{Mar16}.
For more references and an extended discussion, check Bochi, Viana~\cite{BcV04} and
Bonatti, Díaz, Viana~\cite[Section~12.5]{Beyond}.

An alternative approach which has been particularly fruitful in recent years is to deal with the
phenomenon of zero Lyapunov exponents by means of the so-called Invariance Principle,
formulated by Ledrappier~\cite{Led86} and Bonatti, Gomez-Mont, Viana~\cite{BGV03} for linear cocycles,
and Avila, Viana~\cite{Extremal} in the full non-linear setting.
Roughly speaking, the Invariance Principle asserts that systems whose Lyapunov exponents
do vanish must satisfy certain stringent conditions that can often be excluded a priori
for different reasons, for instance, topological.
Among the main applications, let us mention Hertz, Hertz, Tahzibi, Ures~\cite{HHTU12},
Avila, Viana, Wilkinson~\cite{AVW15,AVW22}, and Viana, Yang~\cite{ViY13}.
Recently, Butler, Xu~\cite{BuX18} studied the Lyapunov exponents of partially hyperbolic
diffeomorphisms along the stable (respectively unstable) bundle, finding conditions under
which the extremal exponents coincide.

Partial hyperbolicity (see Section~\ref{s.partially_hyperbolic} for definitions)
provides a particularly convenient context for  studying the persistence of zero Lyapunov exponents.
On the one hand, in that setting one needs only consider the system's Lyapunov exponents along
the centre bundle, $E^c$, as the exponents along the stable and unstable bundles, $E^s$ and $E^u$,
are clearly bounded from zero. On the other hand, the geometric structure exhibited by
partially hyperbolic systems, especially their admitting invariant stable and unstable foliations
tangent to $E^s$ and $E^u$, respectively, makes them particularly suited to the  approaches
mentioned before, particularly to the application of the Invariance Principle.

In this paper we deal with \emph{fibred partially hyperbolic systems} in the sense of
Avila, Viana, Wilkinson~\cite[Section~3.2]{AVW22} (see also Avila, Viana~\cite[Section~6]{AvV20}).
By this we mean that the diffeomorphism $f:N \to N$ is partially hyperbolic and there exists an
$f$-invariant continuous fibre bundle  $\pi:N\to B$ whose fibres are $C^1$ sub-manifolds tangent
to the centre bundle $E^c$ of $f$.
In all the cases we shall consider here, the fibres will be modelled after a compact orientable surface $S$.
Moreover, the diffeomorphism will be taken to be centre-bunched and to preserve some measure in
the Lebesgue class of $N$.

Our first conclusion is that the topology of the fibre $S$ has profound implications on the dynamics of
such maps and, in particular, on their Lyapunov exponents.
Indeed, in Theorem~\ref{theorem_general} we find that if the genus of $S$ is strictly greater
than $1$  then the centre Lyapunov exponents of $f$ must be distinct.
In particular, symplectic fibred partially hyperbolic systems are necessarily hyperbolic unless $S=\SS^2$
or $S=\TT^2$. A similar phenomenon was highlighted by Avila, Viana~\cite[Theorem~6.6]{Extremal}
in a different setting.

The arguments in the proof of Theorem~\ref{theorem_general} also lead to information on the dynamics of
the exceptional diffeomorphisms for which the centre Lyapunov exponents coincide.
For this we restrict ourselves to a subclass of fibred partially hyperbolic systems, namely,
perturbations of certain partially hyperbolic skew-product maps.
By a \emph{partially hyperbolic skew-product} we mean a partially hyperbolic centre-bunched
diffeomorphism of the form 
\begin{equation}\label{eq.skew}
f_0:M \times S \to M \times S, \quad (x,v) \mapsto (g_0(x),\Gamma_x(v))
\end{equation}
where $g_0:M\to M$ is a transitive Anosov diffeomorphism on a compact manifold $M$,
each $\Gamma_x:S \to S$ is a diffeomorphism on a compact orientable surface $S$, 
and the centre bundle of $g$ coincides with the vertical direction $\{0\} \times TS \subset T(M\times S)$.
It is assumed that $f_0$ preserves a measure $\mu$ in the Lebesgue class of $M\times S$,
and the results that follow concern the dynamics of nearby $\mu$-preserving maps.

In Theorem~\ref{theorem_torus} we take $S$ to be the torus $\TT^2$,
and $f_0:M\times\TT^2 \to M \times \TT^2$ to be an \emph{elliptic affine extension} of $g_0$,
that is, a diffeomorphism of the form
\begin{equation}\label{eq.affine_extension}
f_0(x,v) = \left(g_0(x), L_0 v+w_0(x)\right),
\end{equation}
where $L_0$ is any elliptic element of $\SL(2,\ZZ)$.
We prove that every nearby accessible $\mu$-preserving map $f$ which is not hyperbolic must be
topologically conjugate to an elliptic affine extension of $g_0$,
unless it admits some invariant line field or some pair of transverse line fields.
We also check that the latter alternative may be excluded in some cases (Corollary~\ref{corollaryD}),
but not always (Example~\ref{ex.linefield}).

In Theorem~\ref{theorem_sphere} we take $S$ to be the sphere $\SS^2$,
and $f_0:M \times \SS^2 \to M \times \SS^2$ to be a \emph{Möbius extension} of $g_0$, that is,
a diffeomorphism of the form
\begin{equation}\label{eq.Mobius_extension}
f_0(x,v) = (g_0(x),\zeta_x(v))
\end{equation}
where each $\zeta_x(v) = (a_xv+b_x)/(c_xv+d_x)$ is a Möbius transformation, viewed as a 
diffeomorphism of the sphere via the stereographic projection $P:\SS^2 \to \CC\cup\{\infty\}$.
Then we prove that if a nearby accessible $\mu$-preserving map $f$ is not hyperbolic then
it must be topologically conjugate to a Möbius extension of $g_0$.

This manuscript is organised as follows. In Section~\ref{s.statement} we give the precise statements
of all these results, including the formal definitions of the notions involved.
Section~\ref{s.preliminaries} collects a few main tools from the literature that are used in the proofs of our results.
Theorem~\ref{theorem_general} and Corollary~\ref{corollary_general} are proved in Section~\ref{s.theorem_general}.
In Section~\ref{s.density_accessibility} we check that the assumptions of Theorem~\ref{theorem_torus}
are very common. The material in there is not used in the proof, and may be skipped at a first reading.
Theorem~\ref{theorem_torus} is proved in Section~\ref{s.theorem_torus} and 
the proof of Corollary~\ref{corollaryD} is given in Section~\ref{s.line_fields}.
Theorem~\ref{theorem_sphere} is proved in Section~\ref{s.theorem_sphere}.

\paragraph{Acknowledgements} This paper corresponds to the first author's doctoral thesis,
presented at IMPA in late 2021. We are most grateful to committee members L. Backes, L. Lomonaco,
K. Mar{\'\i}n, K. War and J. Yang for numerous comments and suggestions that greatly helped 
improve the presentation. A discussion with J.V. Pereira is also gratefully acknowledged.

\section{Definitions and statement of results}\label{s.statement}

This section contains the statements of our main results, including the precise definitions of the notions involved.

\subsection{Partially hyperbolic diffeomorphisms}\label{s.partially_hyperbolic}

A diffeomorphism $f:N\to N$ on a compact manifold $N$ is \emph{partially hyperbolic} if the tangent space
$TN$ admits a non-trivial $Df$-invariant continuous splitting $TN=E^u\oplus E^c\oplus E^s$ such that:
\begin{itemize}
\item there are positive continuous functions $\nu$, $\widehat{\nu}$, $\gamma$,
and $\widehat{\gamma}$ on $M$ satisfying 
\[
\nu < 1 < \widehat{\nu}^{-1} \text{ and } \nu<\gamma<\widehat{\gamma}^{-1}<\widehat{\nu}^{-1};
\]
\item there is a Riemannian norm $\|\cdot\|$ on $M$ such that for any unit vector $v_p\in T_p(M)$, we have
\begin{equation}\label{eq.partially_hyperbolic}
\begin{array}{cc}
\|Df_p(v_p)\| < \nu(p) & \text{if } v_p\in E^s(p) \\
\gamma(p) < \|Df_p(v_p)\| < \widehat{\gamma}(p)^{-1} & \text{if } v_p\in E^c(p)\\
\widehat{\nu}(p)^{-1} < \|Df_p(v_p)\| & \text{if } v_p\in E^u(p).
\end{array}
\end{equation}
\end{itemize}
We call $f$ \emph{centre-bunched} if the functions $\nu$, $\widehat\nu$, $\gamma$, and 
$\widehat\gamma$ can be chosen to satisfy
\begin{equation}\label{eq.center_bunched}
\nu<\gamma\widehat{\gamma} \text{ and }\widehat{\nu}<\gamma\widehat{\gamma}
\end{equation}
at every point.

$E^u$ and $E^s$ are called, respectively, the \emph{unstable bundle} and the \emph{stable bundle} of $f$.
There exist $f$-invariant foliations $\cF^u$ and $\cF^s$ tangent to $E^u$ and $E^s$, respectively,
at every point. Moreover, both the \emph{unstable foliation} $\cF^u$ and the \emph{stable foliation} 
$\cF^s$ are unique. See~\cite{BP74,HPS77}.
$E^c$ is called the \emph{centre bundle} of $f$. It need not be integrable, in general.
One calls $f$ \emph{dynamically coherent} if there exist $f$-invariant foliations $\cF^{cu}$
and $\cF^{cs}$ tangent at every point to $E^{u}\oplus E^c$ and $E^{c}\oplus E^s$, respectively.
Intersecting their leaves, one obtains an $f$-invariant \emph{centre foliation} $\cF^c$
tangent to $E^c$ at every point. See~\cite{GPS94,BuW10}.
A partially hyperbolic diffeomorphism $f:N\to N$ is said to be \emph{accessible} if any
two points $p$ and $q$ in $N$ are connected by some \emph{$su$-path} in $N$, that is,
some piecewise differentiable oriented curve each of whose (finitely many) legs is
contained in a single leaf of either $\cF^u$ or $\cF^s$.

\subsection{Lyapunov exponents}\label{s.lyapunov}

The \emph{Lebesgue class} of a compact manifold $N$ is the set of measures on $N$ equivalent to one
given by the integration of any volume form (this does not depend on the choice of the form).
We call a diffeomorphism $f:N \to N$ \emph{volume-preserving} if it preserves some measure $\mu$
in the Lebesgue class, and we call it \emph{$\mu$-preserving} if we want to specify that measure.

The theorem of Oseledets~\cite{Ose68} asserts that for $\mu$-almost every point $p\in N$ there
exist $k(p)\in\NN$, real numbers $\lambda_1(p)>\lambda_2(p)>\cdots>\lambda_{k(p)}(p)$ and a
$Df$-invariant splitting $T_p M=E^1_p \oplus E^2_p \oplus \cdots \oplus E^{k(p)}_p$ such that
\[
\lim_{|n|\to\infty}\frac{1}{n}\log \|Df^n_p(v_p)\|=\lambda_j(p)\text{ for all non-zero } v_p\in E^j_p.
\]
The numbers $\lambda_j(p)$ are called the \emph{Lyapunov exponents} and the 
$E^j_p$ are the \emph{Oseledets spaces} of $f$ at $p$.
When the system $(f,\mu)$ is ergodic the functions 
$p \mapsto \kappa(p), \lambda_j(p), \dim E^j_p$ are constant on a full $\mu$-measure set.
We use $\theta_1(p) \ge \cdots  \ge \theta_{\dim N}(p)$ to denote the Lyapunov exponents
counted with multiplicity: the \emph{multiplicity} of each $\lambda_j(p)$ is the dimension of
the subspace $E^j_p$, and so this means that we have $\theta_i(f,p) = \lambda_j(f,p)$ for $\dim E_p^j$
different values of $i$. The map $f$ is said to be \emph{non-uniformly hyperbolic}
for $\mu$ if the set of points where the Lyapunov exponents are all non-zero has
full $\mu$-measure. 

A \emph{symplectic form} on the manifold $N$ is a closed non-degenerate differential $2$-form $\omega$ on $N$.
Such a form exists only if the dimension of $N$ is even,
and then $\omega^{\dim N/2}$ is a volume form on $N$. The associated volume measure will be denoted as $\volume$. A diffemorphism $f:N \to N$ is said
to be \emph{$\omega$-symplectic} if it preserves $\omega$ and, consequently,
the measure $\volume$. In this setting, the Lyapunov exponents have the
following symmetry property: at $\volume$-almost every point,
\begin{equation}\label{eq.symmetric}
\theta_j(f,q)+\theta_{\dim N+1-j}(f,q) = 0 \text{ for every } j=1, \dots, \dim N/2.
\end{equation}

If $f:N\to N$ is partially hyperbolic and symplectic then
(see \cite[Section~4]{BcV04})
$\dim E^u=\dim E^s$ and $E^u\oplus E^s$ coincides with the symplectic orthogonal
of the centre bundle $E^c$:
$$
\omega_p(u_1,u_2)=0 \text{ for all } u_2 \in E^c_p \quad\Leftrightarrow\quad u_1 \in E^u_p \oplus E^s_p.
$$
In particular, the restriction $\omega^c = \left.\omega\right|_{E^c}$ is non-degenerate at every point.
Thus, assuming the diffeomorphism is dynamically coherent, $\omega^c$ defines a symplectic form 
(an area form) on each centre leaf. Clearly, this symplectic structure is preserved by the restriction
of $f$ to centre leaves.

\subsection{Fibred partially hyperbolic systems}\label{s.fibred}

Recall that we call \emph{fibred partially hyperbolic system} on a compact manifold $N$ any partially
hyperbolic diffeomorphism $f:N \to N$ such that there exists an $f$-invariant continuous fibre bundle
$\pi:N\to B$ whose fibres are $C^1$ sub-manifolds tangent to the centre bundle $E^c$ of $f$.
Then $f$ is dynamically coherent, with the fibration as a centre foliation.
In all the cases we consider, the fibres are modelled after a compact orientable surface $S$.

Given $r\ge 1$ and a measure $\mu$ in the Lebesgue class, we shall denote by $\cF^r_\mu(N,S)$
the space of all $\mu$-preserving centre-bunched $C^r$ fibred partially hyperbolic systems.
Analogously, given $r\ge 1$ and a symplectic form $\omega$, we shall denote by
$\cF^r_\omega(N,S)$ the space of all $\omega$-symplectic centre-bunched $C^r$ fibred partially
hyperbolic systems. It is clear that $\cF^r_\omega(N,S)\subset\cF^r_\mu(N,S)$ if $\mu=\volume$
is the volume measure induced by $\omega$.

\begin{example}\label{ex.skew_product}
Let $f_0:M \times S \to M \times S$ be a $C^r$ partially hyperbolic skew-product,
as defined in \eqref{eq.skew}. If $f_0$ preserves a probability measure $\mu$ in the 
Lebesgue class of $M\times S$ then $f_0$ is in the interior of $\cF^r_\mu(M\times S,S)$
among all $\mu$-preserving maps. Analogously, if $f_0$ preserves a symplectic form $\omega$
on $M \times S$ then $f_0$ is in the interior of $\cF^r_\omega(M\times S,S)$ among all
$\omega$-symplectic maps.

Indeed, let $f:M\times S\to M\times S$ be any diffeomorphism in a small $C^r$-neighbourhood
of $f_0$. By normal hyperbolicity theory (see \cite[Theorem~4.1]{HPS77}), $f$ is partially
hyperbolic and dynamically coherent, with a centre foliation $\cF^c$ whose leaves are
uniformly $C^r$-close to the leaves of the centre foliation
$\left\{\{x\}\times S: x \in M\right\}$ of the unperturbed diffeomorphism $f_0$. 
Moreover (see \cite[Theorems~7.1 and~7.4]{HPS77}), the two centre foliations are conjugate,
in the sense that there exists a homeomorphism $H_f:M\times S \to M \times S$ that sends
each $\{x\}\times S$ to a centre leaf $\cF^c_x$ of $f$, in such a way that
\begin{equation}\label{eq.leaf_conjugacy}
f \left(H_f(\{x\}\times S)\right) = H_f \left(f_0(\{x\}\times S)\right).
\end{equation}
The \emph{leaf conjugacy} $H_f$ is not unique, but the correspondence $x \mapsto \cF^c_x$
defined in this way does not depend on the choice of $H_f$. Moreover, the restriction of
$H_f$ to each $\{x\}\times S$ is a $C^1$ diffeomorphism onto $\cF^c_x$, and these leaf
derivatives vary continuously on $M \times  S$. See \cite{PSW12}.
In particular, the leaves of $\cF^c$ are the fibres of a continuous fibre bundle
$\pi:M \times  S \to M$. Since centre-bunching is a $C^1$-open condition, we also have
that every diffeomorphism $f$ close to $f_0$ is centre-bunched. This proves the claim.

Related remarks were made in \cite[Section~3.2]{AVW22}, \cite[Example~2.7]{ASV13}, and \cite[Section~3]{Mar16}.  We also point out that, by \cite[Theorem~6.1]{AvV20},
stably accessible diffeomorphisms are dense in a neighbourhood of $f_0$.
\end{example}

\begin{mainthm}\label{theorem_general}
If $\mu$ is a measure in the Lebesgue class of $N$ then every accessible $f \in \cF_\mu^r(N,S)$,
$r\ge 2$ satisfies at least one of the following conditions:  
\begin{itemize}
\item[(1)] the centre Lyapunov exponents of $f$ are distinct,
    and they are continuous at $f$ as functions of the dynamical system;
\item[(2)] $S=\TT^2$ and the centre bundle $E^c$ of $f$ contains an invariant continuous line
    field or an invariant pair of transverse continuous line fields;
    \item[(3)] $S=\SS^2$ or $S=\TT^2$ and there exists a continuous conformal structure on the
    centre leaves invariant under $f$ and under its stable and unstable holonomies.
\end{itemize}
\end{mainthm}

A few comments are in order concerning the three alternatives in the conclusion of the theorem.
Firstly, under additional
assumptions (pinching, existence of periodic points) one can ensure that every diffeomorphism in
$\cF_\mu^r(N,S)$ or $\cF_\omega^r(N,S)$ is $C^r$-approximated by one that satisfies condition (1).
This follows from the methods developed in \cite[Theorems~A~and~B]{Mar16} and \cite[Corollary~1]{LMY18}
but we shall not detail the arguments here.
Our next main goal will be to characterise the third alternative in the theorem more precisely.
Indeed, we shall see in Sections~\ref{s.torus} and~\ref{s.sphere} that diffeomorphisms as in (3)
are quite rigid. 
For that we shall restrict our setting somewhat in each of the two cases $S=\SS^2$ and $S=\TT^2$.
Example~\ref{ex.linefield} below shows that alternative (2) may also occur.
At this point, it is not clear whether a rigidity statement holds in that setting as well.
But in Section~\ref{s.line_fields} we show how this alternative can be excluded in some cases.

Theorem~\ref{theorem_general} applies, in particular, to the volume measure $\volume$ associated
to any given symplectic form $\omega$.
Besides, in the symplectic case the conclusions are a bit stronger:

\begin{maincor}\label{corollary_general}
If $f\in\cF_\omega^r(N,S)$ for some symplectic form $\omega$, alternative (1) in
Theorem~\ref{theorem_general} implies that $f$ is non-uniformly hyperbolic for the invariant
measure $\volume$, and alternative (3) implies that there exists a continuous Riemannian metric
on the centre leaves invariant under $f$ and under the stable and unstable holonomies.
\end{maincor}

\subsection{Rigidity - the torus case}\label{s.torus}

Initially, we deal with the case $S=\TT^2$.
Let $f_0:M\times\TT^2\to M \times \TT^2$ be an elliptic affine extension of an
Anosov diffeomorphism $g_0:M \to M$, as defined in \eqref{eq.affine_extension}.
We speak of an \emph{$L_0$-affine extension} instead when we want to specify
the choice of the elliptic element $L_0$ of $\SL(2,\ZZ)$.
In Section~\ref{s.accessibility} we observe that stable accessibility is dense among the diffeomorphisms of the form \eqref{eq.affine_extension}.

\begin{remark}\label{r.finite_order}
Every elliptic $L_0\in\SL(2,\ZZ)$ has finite order, which must be $1$, $2$, $3$, $4$ or $6$.
Indeed, let $e^{\pm i\theta}$ be the eigenvalues.
The characteristic equation $x^2 - \trace(L_0) x + 1 = 0$ gives that $\trace(L_0)=2\cos\theta$.
As the trace is an integer, it follows that $\cos\theta\in\{-1,-1/2,0,1/2,1\}$,
which implies the claim.
\end{remark}

We also assume that $f_0$ preserves a given measure $\mu$ in the Lebesgue class of $M \times\TT^2$.
That is the case, for example, if $f$ preserves a symplectic form $\omega$ on $M \times\TT^2$ and,
more specifically, when the Anosov diffeomorphism $g_0$ itself preserves a symplectic form $\omega_M$ on
$M$: then it suffices to take $\omega=\omega_M \times \omega_S$, where
$\omega_S$ is the standard area form on $\TT^2$.

\begin{mainthm}~\label{theorem_torus}
There exists a neighbourhood $\cU_T$ of $f_0$ in the space of $C^r$, $r \ge 2$
diffeomorphisms of $M\times\TT^2$ such that every accessible $\mu$-preserving
$f\in\cU_T$ satisfies at least one of the following conditions:
\begin{enumerate}
\item[(1)] the centre Lyapunov exponents of $f$ are distinct,
    and they are continuous at $f$ as functions of the dynamical system;
\item[(2)] the centre bundle $E^{c}$ contains an invariant continuous line field or an invariant pair of transverse continuous line fields;
\item[(3)] $f$ is topologically conjugate to an $L_0$-affine extension $(x,v) \mapsto (g_0(x), L_0v+w(x))$ of $g_0$.
\end{enumerate}
\end{mainthm}

In the symplectic case, the alternative (1) implies that $f$ non-uniformly hyperbolic.
Note that the $L_0$-elliptic extension in the alternative (3) is only claimed to be a homeomorphism,
as our methods can only prove that $w$ is continuous
(H\"older regularity of the solutions of the Beltrami equation, cf. \cite[Theorem~8]{AhB60},
can probably be used to prove that $w$ is H\"older).
It would be interesting to know whether this can be upgraded to differentiability.

Another interesting open question concerning alternative (3) of the theorem
is whether the conjugacy may be taken to be differentiable,
at least when $M$ is a surface. That would imply that the $L_0$-elliptic
extension is differentiable, of course.
The ideas in \cite[Section~7.3]{AVW15} suggest that progress is perhaps 
possible also on the way of the converse.

In the corollary that follows it is assumed that $g_0$ has some fixed point.
By a result of
Sondow~\cite{Son76} (see also Franks~\cite{Fra69},
Hirsch~\cite{Hir71} and Manning~\cite{Mann73,Mann74}),
that is automatic if $M$ is the quotient $G/F$ of a compact, connected Lie group $G$ by any finite subgroup $F$.
In particular, the assumption is automatically satisfied if $M=\TT^d$ for any $d\ge 2$. 

\begin{maincor}~\label{corollaryD}
In the setting of Theorem~\ref{theorem_torus}, assume that $g_0$ has a fixed
point and the order of $L_0$ is greater than $2$.
Then, up to reducing the neighbourhood $\cU_T$ if necessary, every accessible
$\mu$-preserving $f\in\cU_T$ satisfies that either the two centre Lyapunov
exponents are distinct, or $f$ is topologically conjugate to an $L_0$-affine
extension $(x,v) \mapsto (g_0(x), L_0v+w(x))$ of $g_0$.
\end{maincor}

In other words, alternative (2) of Theorem~\ref{theorem_torus} may be excluded
when the order of $L_0$ is greater than $2$. Example~\ref{ex.linefield} shows
that this need no longer be true when $L_0=\pm\Id$.

\subsection{Rigidity - the sphere case}\label{s.sphere}

Now let us consider the case $S=\SS^2$.
Take $f_0:M \times \SS^2 \to M \times \SS^2$ to be a Möbius extension of
a transitive Anosov diffeomorphism $g_0:M\to M$, as defined in \eqref{eq.Mobius_extension}.
It is assumed that $f_0$ preserves a given measure $\mu$ in the Lebesgue
class of $M \times\SS^2$.

\begin{mainthm}~\label{theorem_sphere}
There exists a neighbourhood $\cU_S$ of $f_0$ in the space of $C^r$, $r \ge 2$ diffeomorphisms of $M\times\SS^2$ such that every accessible $\mu$-preserving
$f\in\cU_S$ satisfies one of the following conditions: 
\begin{enumerate}
    \item[(1)] the centre Lyapunov exponents of $f$ are distinct,
    and they are continuous at $f$ as functions of the dynamical system;
    \item[(2)] $f$ is topologically conjugate to some Möbius extension of $g_0$.
\end{enumerate}
\end{mainthm}

In the symplectic case, the alternative (1) implies that $f$ is non-uniformly hyperbolic. 
Note that requiring $f_0$ to be symplectic imposes strong restrictions on the Möbius transformations $\zeta_x$:

\begin{example}\label{ex.sphere}
Let $g_0:M\to M$ be an Anosov diffeomorphism preserving some symplectic form $\omega_M$ on $M$.
Let $\omega=\omega_M\times\omega_S$, where $\omega_S$ is the standard area
form on the sphere, induced by the Euclidean volume form in $\RR^3$ through
$$
\omega_{S,p}(u,v) = (dx_1 \wedge dx_2 \wedge dx_3)\big(u,v,p\big)
$$
for any $p \in \SS^2$ and
$u, v \in T_p\SS^2$. It is clear that a map $f_0$ as in \eqref{eq.Mobius_extension} preserves $\omega$
if and only if every $\zeta_x$ preserves $\omega_S$. We claim that the latter happens if and only
if  $\zeta_x$ is a rotation, that is, the restriction of a rigid motion of $\RR^3$
that preserves the unit sphere. Since the `if' part is obvious, we only have to check that if
$\zeta_x$ preserves $\omega_S$ then it is a rotation. That can be done as follows.

Let $P:\SS^2 \to \CC \cup \{\infty\}$ be the stereographic projection.
As observed by Arnold, Rogness~\cite{ArR08} and Siliciano~\cite{Sil12},
for every Möbius transformation $\zeta:\SS^2\to\SS^2$ there exists a
unique rigid motion $T:\RR^3 \to \RR^3$ mapping $\SS^2$ to a sphere
$T(\SS^2)$ whose ``north pole'' lies in the upper half-space
-- this ensures that the steorographic projection
$Q:T(\SS^2) \to \CC \cup \{\infty\}$ is well-defined -- such that 
$$
\zeta = P^{-1} \circ Q \circ T:\SS^2 \to \SS^2.
$$
It follows that $\zeta$ is area-preserving if and only if the map
$P^{-1} \circ Q:T(\SS^2) \to \SS^2$ is area-preserving, meaning
that its Jacobian $J(P^{-1} \circ Q)$ with respect to the standard
area-forms on the two spheres is constant equal to $1$.

The Jacobian of $P^{-1} : \CC \cup \{\infty\} \to \SS^2$ with respect to
the standard area forms on the plane and the sphere is
$$
JP^{-1}(z) = \frac{4}{(1+|z|^2)^2}.
$$
Thus the level sets of $JP^{-1}$ are the circles about the origin,
and the level sets of $JP$ are the parallels of $\SS^2$, that is,
the intersections of the sphere with the horizontal planes.
Analogously, the level sets of $JQ^{-1}$ are the circles about
the point $x_0+iy_0\in\CC$, where $(x_0,y_0,z_0)=T(0,0,0)$ is the centre
of $T(\SS^2)$ (keep in mind that $z_0>-1$).
For the Jacobian of $P^{-1} \circ Q$ to be constant, the level sets must
coincide, and so we must have $x_0=y_0=0$. Observing that 
$$
JP^{-1}(0) = 4 \text{ and } JQ^{-1}(0) = \frac{4}{(1+z_0)^2},
$$
we see that we must also have $z_0=0$.
Then $T(\SS^2)=\SS^2$ and so $Q=P$, showing that $\zeta = T$ is a rotation.
\end{example}

\section{Preliminaries}\label{s.preliminaries}

In this section we collect a few tools from the literature that will be used in our arguments.

\subsection{Conformal barycentres}\label{s.conformal_barycentre}

Recall that a \emph{conformal structure} on a vector space $V$ is an inner product up to multiplication
by a positive scalar or, more formally, a projective class of inner products on $V$.
\emph{We consider only real vector spaces and linear maps.}
We call \emph{canonical} the conformal structure on $\CC=\RR^2$ associated to the Euclidean inner product.

Let $\HH\subset\CC$ be the Poincaré upper half-plane.
A construction of Douady, Earle~\cite[Section~2]{DE86} associates to each probability measure
$m$ on the boundary $\partial \HH$ with no atom of mass $1/2$ or greater a \emph{conformal barycentre}
$B(m)\in\HH$ which is invariant with respect to the conformal automorphisms of the half-plane:
\begin{equation}\label{eq.equivariance}
B(\phi_* m) = \phi(B(m)) \text{ for every } \phi \in \PSL(2,\RR).
\end{equation}
We associate to $m$ the unique conformal structure on $\CC$ which is preserved by the stabiliser
$\Stab(B(m))$ of the conformal barycentre, that is, which is invariant under every linear isomorphism
$$
A = \left(\begin{array}{cc} a & b \\ c & d \end{array}\right):\CC\to\CC
$$
such that $\phi_A(B(m))=B(m)$ for the automorphism $\phi_A \in \PSL(2,\RR)$ defined by
$$ 
\phi_A(z) = \frac{az+b}{cz+d}.
$$

\begin{remark}
The stabiliser of the imaginary unit $i$ is the subgroup of automorphisms $\phi\in\PSL(2,\RR)$ of the form
$$
\phi(z) = \frac{az+b}{cz+d}\quad \text{with $a = d$ and $b=-c$}.
$$
The corresponding linear isomorphisms $A:\CC\to\CC$ preserve the standard conformal structure
on $\CC$ and, moreover, that is the only conformal structure on the plane invariant under every 
linear isomorphism $A$ such that $\phi_A\in\Stab(i)$.
The stabiliser of any other $w\in\HH$ coincides with $\phi_w\Stab(i)\phi_w^{-1}$,
where $\phi_w\in\PSL(2,\RR)$,
$$
\phi_w(z) = \frac{a_wz+b_w}{c_wz+d_w},
$$
is such that $\phi_w(i)=w$.
Hence $\Stab(w)$ also preserves a unique conformal structure on $\CC$, namely,
the push-forward of the standard conformal structure under the linear isomorphism
$$
\left(\begin{array}{cc}a_w & b_w \\ c_w & d_w \end{array}\right):\CC \to \CC.
$$
\end{remark}

\begin{remark}\label{r.continuity}
It follows from the construction in~\cite[Section~2]{DE86} that the conformal barycentre
varies continuously with the probability measure relative to the weak$^*$ topology.
Indeed, for each $m$, consider the vector field $\xi_m:\HH\to\RR$ defined by
$$
\xi_m(w) = \frac{1}{\phi_w'(w)} \int_{\partial\HH} \phi_w(z) \, dm(z)
$$
where $\phi_w$ is the conformal automorphism of the half-plane such that $\phi_w(w)=i$.
It is clear that $\xi_m$ is real-analytic and varies continuously with $m$ relative to the weak$^*$ topology.
By construction, the conformal barycentre $B(m)$ is the only zero of $\xi_m$,
and it has index $1$, meaning that the winding number of $\xi_m$ along any small simple
closed curve $c$ around $B(m)$ is equal to $1$. Fix an arbitrarily small $c$.
The winding number of $\xi_{m'}$ along $c$ remains $1$ for any $m'$ close to $m$,
and that implies that $\xi_{m'}$ has some zero in the inside of $c$.
By uniqueness, this means that $B(m')$ is in the inside of $c$, for any $m'$ close to $m$,
which proves the claim.
\end{remark}

\subsection{Measurable Riemann mapping theorem}

We quote a few useful facts from conformal mapping theory.
See Ahlfors, Bers~\cite{AhB60} and Ahlfors~\cite{Ah66,Ah73} for more details and references.

A map $f$ on a Riemannian manifold is said to be \emph{conformal} if the 
derivative at (almost) every point preserves angles, relative to the given Riemannian metric.
For any domain $U$ of the plane, the Riemannian metric may always be written
as $ds = \lambda |dz + \mu d\bar{z}|$ where $\lambda =\lambda(z)$ is a positive number and
$\mu=\mu(z)$ is a complex number with $|\mu| < 1$.
Then conformality with respect to this metric (called \emph{$\mu$-conformality})
is expressed by the \emph{Beltrami equation}
\begin{equation}\label{eq.Beltrami1}
\partial_{\bar{z}} f = \mu \, \partial_z f,
\end{equation}
where
$$
\partial_z f = \frac 12 (\partial_x f - i \partial_y f)
\text{ and } 
\partial_{\bar{z}} f = \frac 12 (\partial_x f + i \partial_y f)
$$

The measurable Riemann mapping theorem asserts that if $\mu$ is measurable
and $\sup|\mu| < 1$ then a $\mu$-conformal map $f$ does exist.
More precisely, we shall use the following existence and uniqueness result,
which is contained in \cite[Theorem~6]{AhB60}:

\begin{theorem}\label{t.mRmt}
Let $\mu:\CC \to \CC$ be a measurable function such that $\sup|\mu|<1$.
Then there exists a unique homeomorphism $f:\CC \to \CC$ which is $\mu$-conformal and
fixes the points $0$, $1$ and $\infty$.
\end{theorem}

We shall also need the fact that the homeomorphism $f$ depends continuously
on $\mu$ in the sense of the following result, which is contained in 
\cite[Theorem~8]{AhB60}:

\begin{theorem}\label{t.mRmt_continuity}
Let $k<1$ be fixed. Then for any compact set $K\subset\CC$ there exists $C(K)>0$ such that 
$$
\sup_{z\in K} |f_1(z)-f_2(z)| \le C(K) \sup |\mu_1 - \mu_2|
$$
for any measurable functions
$\mu_1, \mu_2:\CC \to \CC$ with $\sup|\mu_i|<k$ for $i=1,2$,
where $f_i:\CC \to \CC$ denotes the $\mu_i$-conformal homemorphism that
fixes the points $0$, $1$, and $\infty$ (given by Theorem~\ref{t.mRmt}).
\end{theorem}

\subsection{Invariance Principle}\label{s.principle}

Next we recall a few useful ideas from Ledrappier~\cite{Led86}, Bonatti, Gomez-Mont, Viana~\cite{BGV03},
Avila, Viana~\cite{Extremal} and Avila, Santamaria, Viana~\cite{ASV13}.
Additional related information can be found in Viana~\cite{LLE}.

\subsubsection{Cocycles and exponents}\label{ss.exponents}

Let $f:M \to M$ be a partially hyperbolic diffeomorphism on a compact manifold
and $\pi:\cV \to M$ be a continuous finite-dimensional vector bundle.
A \emph{linear cocycle} over $f$ is a continuous map $F:\cV \to \cV$ such that
$\pi \circ F = f \circ \pi$ and $F$ acts on the fibres by
linear isomorphisms $F_x: \cV_x \to \cV_{f(x)}$.
A theorem of Furstenberg, Kesten~\cite{FK60} asserts that the
\emph{extremal Lyapunov exponents} 
$$
\lambda_+(F,x) = \lim_n \frac 1n \log \|F_x^n\| 
\text{ and } 
\lambda_-(F,x) = \lim_n - \frac 1n \log \|(F_x^n)^{-1}\| 
$$
exist at $\mu$-almost every point $x \in M$, relative to any $f$-invariant probability measure $\mu$.
Note that $\lambda_-(F,x) \le \lambda_+(F,x)$ whenever they are defined.
If the system $(f,\mu)$ is ergodic then the functions $x \mapsto \lambda_\pm(F,x)$
are constant on a full $\mu$-measure set.

The \emph{projectivisation} of $\pi:\cV\to M$ is the continuous fibre bundle $\pi:\PP\cV \to M$
whose fibres are the projective quotients of the fibres of $\cV$.
Note that the fibres are smooth manifolds modelled after 
some projective space $\PP\RR^k$.
The \emph{projective cocycle} associated to $F:\cV \to \cV$ is the smooth cocycle
$\PP F : \PP\cV \to \PP \cV$ whose action $\PP F_x : \PP\cV_x \to \PP\cV_{f(x)}$ on the fibres
is given by the projectivisation of $F_x$.
The \emph{extremal Lyapunov exponents} of $\PP F$ are the numbers
$$
\lambda_+(\PP F,x,\xi) = \lim_n \frac 1n \log \|DF_x^n(\xi)\| 
\text{ and } 
\lambda_-(F,x,\xi) = \lim_n - \frac 1n \log \|DF_x^n(\xi)^{-1}\|.
$$
They are defined at $m$-almost every point $(x,\xi) \in \PP\cV$, for any $\PP F$-invariant
measure $m$. Moreover,
$$
\lambda_+(\PP F, x, \xi) \le \lambda_+(F,x) - \lambda_-(F,x)
\text{ and } 
\lambda_-(\PP F, x, \xi) \ge \lambda_-(F,x) - \lambda_+(F,x)
$$
whenever they are defined.

\subsubsection{Invariant holonomies}\label{ss.holonomies}

We call a \emph{stable holonomy} for $\PP F$ a family of homeomorphisms $H^s_{x,y}:\PP\cV_x \to \PP\cV_y$
defined for all $x$ and $y$ in the same stable leaf of $f$ and satisfying
\begin{itemize}
\item[(a)] $H^s_{y,z} \circ H^s_{x,y} = H^s_{x,z}$ and $H^s_{x,x} = \id$;
\item[(b)] $\PP F_y \circ H^s_{x,y} = H^s_{f(x),f(y)} \circ \PP F_x$;
\item[(c)] $(x,y,\xi) \mapsto H^s_{x,y}(\xi)$ is continuous when $(x,y)$ varies in the set of pairs of points in the same local stable leaf;
\item[(d)] there are $C>0$ and $\eta>0$ such that $H^s_{x,y}$ is $(C,\eta)$-H\"older for every $x$ and $y$ in the same local stable leaf.
\end{itemize}
The concept of \emph{unstable holonomy} for $\PP F$ is analogous, replacing (local) stable leaves
with (local) unstable leaves.

Stable and unstable holonomies for $\PP F$ do exist when the linear cocycle $F$ is \emph{fibre-bunched},
meaning that for some choice of the Riemannian norm $\|\cdot\|$ we have
\begin{equation}\label{eq.fibre_bunched}
\|F_x\|\|(F_x)^{-1}\| \nu(x) < 1 
\text{ and } 
\|F_x\|\|(F_x)^{-1}\| \hat\nu(x) < 1, 
\end{equation}
where $\nu(\cdot)$ and $\hat\nu(\cdot)$ are continuous functions as in \eqref{eq.partially_hyperbolic}.
It is clear from the definitions \eqref{eq.center_bunched} and \eqref{eq.fibre_bunched} that the
diffeomorphism $f$ is centre-bunched if and only if the centre derivative cocycle $F=\Dcf$ is 
fibre-bunched.

\subsubsection{Invariant disintegrations}\label{ss.disintegrations}

Since $\PP F$ is continuous and the base space $M$ is compact, for any $f$-invariant measure
$\mu$ there exist $\PP F$-invariant measures $m$ with $\pi_* m = \mu$. Fix any such measure $m$.
By the Rokhlin disintegration theorem (see \cite[Chapter~5]{FET16}), there exists a \emph{disintegration}
of $m$ into conditional probabilities along the fibres, that is, a measurable family $\{m_x: x \in M\}$
of probability measures such that $m_x(\PP\cV_x) = 1$ for $\mu$-almost every $x$ and
$$
m(U) = \int_M m_x(U \cap \PP\cV_x) \, d\mu(x)
$$
for every measurable set $U \subset \PP\cV$. The disintegration is essentially unique, meaning that any
two coincide on some full $\mu$-measure subset.

The disintegration is said to be \emph{$s$-invariant} if
$$
\left(H^s_{x,y}\right)_* m_x = m_y
\text{ for every $x$ and $y$ in the same stable leaf.}
$$
The notion of $u$-invariance is analogous, using $u$-holonomy instead.
We say that the disintegration is \emph{bi-invariant} it is both $s$-invariant and $u$-invariant.
A $\PP F$-invariant probability measure $m$ is called an \emph{$su$-state} if it admits a bi-invariant
disintegration.

\subsubsection{Invariance Principle}\label{s.invariance_principle}

A subset of $M$ is \emph{$s$-saturated} if it consists of entire stable leaves,
and \emph{$u$-saturated} if it consists of entire unstable leaves.
Moreover, we call it \emph{bi-saturated} if it is both $s$-saturated and $u$-saturated.

Assuming accessibility, if the Lyapunov exponents $\lambda_\pm(F,\cdot)$ coincide then every
$\PP F$-invariant measure that projects down to $\mu$ is an $su$-state.
That is a consequence of the following version of the Invariance Principle,
which is contained in Theorems B and C of \cite{ASV13}:

\begin{theorem}\label{t.invariance_principle}
Let $f:M \to M$ be a $C^2$ partially hyperbolic centre-bunched diffeomorphism preserving a
measure$\mu$ in the Lebesgue class of $M$.
Let $F:\cV\to\cV$ be a linear cocycle over $f$ admitting invariant stable and unstable holonomies,
and suppose that $\lambda_-(F,x) = \lambda_+(F,x)$ at $\mu$-almost every point.

Then every $\PP F$-invariant probability $m$ on $\PP\cV$ with $\pi_*m=\mu$ admits a
disintegration $\{m_x: x \in M\}$ such that
\begin{itemize}
\item[(a)] the disintegration is bi-invariant over a full-measure bi-saturated subset $M_F$ of $M$;
\item[(b)] if $f$ is accessible then $M_F=M$ and the conditional probabilities $m_x$ depend continuously
on the base point $x \in M$, relative to the weak$^*$ topology.
\end{itemize}
\end{theorem}

Continuity dependence of the conditional probabilities is actually a consequence of bi-invariance,
when the diffeomorphism is accessible, as shown in \cite[Section~7]{ASV13}.

\section{Proof of Theorem~\ref{theorem_general}}\label{s.theorem_general}

The assumption that $f\in\cF^r_\mu(M,S)$ includes that $f$ is partially
hyperbolic, centre-bunched, and dynamically coherent. Since $f$ is also taken
to be accessible, it follows from \cite[Theorem~A]{PSh00} that the system
$(f,\mu)$ is ergodic. In particular, the Lyapunov exponents are constant on
a full measure subset.

\subsection{Invariant holonomies}\label{s.invariant_holonomies}

The hypothesis that $f$ is a fibred partially hyperbolic system also ensures
(see \cite[Section~3.2]{AVW22}) that it has \emph{global} stable and unstable
holonomies: for any $x, y \in M$ such that $\cF^c_x$ and $\cF^c_y$ are
contained in the same centre-stable leaf there exists a homeomorphism
(\emph{stable holonomy map})
\begin{equation}\label{eq.s-holonomy}
h^s_{x,y}:\cF^c_x \to \cF^c_y
\text{ such that }  h_{x,y}^s(p) \in \cF^s(p) \cap \cF^c_y \text{ for all } p \in \cF^c_x,
\end{equation}
and for any $x, y \in M$ such that $\cF^c_x$ and $\cF^c_y$ are contained
in the same centre-unstable leaf there exists a homeomorphism
(\emph{unstable holonomy map})
\begin{equation}\label{eq.u-holonomy}
h^u_{x,y}:\cF^c_x \to \cF^c_y
\text{ such that }  h_{x,y}^u(p) \in \cF^u(p) \cap \cF^c_y \text{ for all } p \in \cF^c_x.
\end{equation}
See also \cite[Example 2.7]{ASV13} and \cite[Section~3]{Mar16}.

By \cite[Theorem B]{PSW97}, the centre-bunching assumption ensures that these
holonomy maps are $C^1$. Their derivatives induce invariant holonomies for the projective
derivative cocycle $\PP(\Dcf): \PP(E^c) \to \PP(E^c)$ acting on the projective centre bundle $\PP(E^c)$:
take
$$
H^s_{p,q}:\PP(E^c_p) \to \PP(E^c_q), \quad H^s_{p,q} = \PP\left((Dh^s_{x,y})_p\right) 
$$
as the stable holonomy maps, for $x, y \in M$ such that $\cF^c_x$ and $\cF^c_y$
are contained in the same centre-stable leaf, $p\in\cF^c_x$ and $q=h^s_{x,y}(p)$,
and
$$
H^u_{p,q}:\PP(E^c_p) \to \PP(E^c_q), \quad H^u_{p,q} = \PP\left((Dh^u_{x,y})_p\right) 
$$
as the unstable holonomy maps, for $x, y \in M$ such that $\cF^c_x$ and $\cF^c_y$
are contained in the same centre-unstable leaf, $p \in \cF^c_x$ and $q = h^u_{x,y}(p)\in\cF^c_y$.

Recall that we call $su$-state any $\PP(\Dcf)$-invariant measure $m$
on $\PP(E^c)$ that projects down to $\mu$ on $M\times S$ and admits a bi-invariant
disintegration $\{m_q : q \in M\times  S\}$ along the fibres of $\PP(E^c)$:
\begin{equation}\label{eq.su-invariance}
\begin{aligned}
\left(H_{q,r}^s\right)_* m_q & = m_r \text{ for any $q,r$ in the same centre-stable leaf}\\
\left(H_{q,r}^u\right)_* m_q & = m_r \text{ for any $q,r$ in the same centre-unstable leaf},
\end{aligned}
\end{equation}

Initially, let us assume that $f$ has no $su$-state.
Then, by the Invariance Principle (Theorem~\ref{t.invariance_principle}(b)),
the two centre Lyapunov exponents must be distinct.
Non-existence of $su$-states also implies that $f$ is a continuity point
for the centre Lyapunov exponents.
That follows from \cite[Proposition~4.8]{LMY18},
which is itself a version of \cite[Proposition~6.3]{Extremal}.
In this way we get alternative (1) in the theorem.

\subsection{Invariant disintegrations}\label{s.disintegrations}

From now on, assume that there does exist some $su$-state $m$,
and let  $\{m_q : q \in M\times  S\}$ be a bi-invariant disintegration along the
fibres of $\PP(E^c)$. As observed previously (see also \cite[Section~7]{ASV13}),
the conditional probabilities $m_q$ depend continuously on the base point $q$ relative
to the weak$^*$ topology. Since $m$ is $\PP(\Dcf)$-invariant, we have that
$$
\PP(\Dcf)_*m_q = m_{f(q)}
\text{ for $\mu$-almost every point $q$.}
$$
This, together with continuity and the fact that $\mu$ is supported
on the whole $M \times  S$, ensures that
\begin{equation}\label{eq.invariance}
\PP(\Dcf)_*m_q = m_{f(q)} \text{ for every } q \in M \times  S.
\end{equation}

Suppose that there exists a point $p \in M \times S$ such that $m_p$ admits
an atom with largest mass, and this mass is greater than or equal to $1/2$.
Then, since $f$ is accessible and the disintegration is $su$-invariant,
the same is true at every point $q\in M \times S$.
Let $v_q$ denote the corresponding atom.
The map $q \mapsto v_q$ defines a continuous line bundle on $M\times S$ tangent to the centre leaves at every point.
By the Poincaré--Hopf theorem this implies that the centre leaves are tori, and so $S = \TT^2$.
Moreover, the properties \eqref{eq.su-invariance} and \eqref{eq.invariance}
give that this line bundle
is $Df$-invariant and invariant under the stable and unstable holonomies. 

Similarly, suppose that some $m_p$ admits a pair of atoms $\{u_p, v_p\}$ with
masses equal to $1/2$.
Then, just as before, the same must hold at every point $q \in M \times S$:
let $\{u_q, v_q\}$ denote the corresponding pair of atoms. 
The map $q \mapsto \{u_q, v_q\}$ defines a pair of transverse continuous line fields.
Again by the Poincaré--Hopf theorem, existence of such a pair implies that $S = \TT^2$.
Finally, the properties \eqref{eq.su-invariance} and \eqref{eq.invariance} ensure that this pair of
transverse continuous line fields is $Df$-invariant and invariant under the stable and unstable holonomies. 

The situations in the previous couple of paragraphs correspond to alternative (2) in the theorem. 

\subsection{Invariant conformal structures}\label{s.invariant_structures}

We are left to consider the case when the conditional probability measures
$m_q$ have no atoms of mass greater than or equal to $1/2$.
We are going to show that, via the conformal barycentre construction in Section~\ref{s.conformal_barycentre}, each $m_q$ determines a unique
conformal structure on the corresponding centre subspace $E^c_q$.
In that way, the leaves of the centre foliation $\cF^c$ are endowed with
Riemann surface structures which are invariant under stable and unstable holonomies and the diffeomorphism itself.
This goes as follows.

The boundary $\partial\HH$ is naturally identified with the \emph{real} projective space
$\PP(\CC)=\PP(\RR^2)$ through the map $x \mapsto [x:1]$.
Consider any linear isomorphism $L:E_q^c \to \CC$ and its projectivisation
$\PP L:\PP(E_c^q) \to \PP(\CC)$. Let $\cM$ be the push-forward of $m_q$ under $\PP L$. 
Consider the conformal structure defined on $\CC$ by the conformal barycentre $B(\cM)$,
and transport it to $E_q^c$ through $L$.
The fact that the barycentre is conformally invariant ensures that the conformal structure
thus defined on $E^c_q$ does not depend on the choice of $L$, as we are going to explain.

Given any other linear isomorphism $L':E_q^c \to \CC$, consider
$$
A = L' \circ L^{-1} = \left(\begin{array}{cc} a & b \\ c & d \end{array} \right):\CC\to\CC.
$$
The projectivisation $\PP A:\PP(\CC) \to \PP(\CC)$ is given in homogeneous coordinates by
$$
[x:1] \mapsto [\phi_A(x) :1],
$$
and so the probability measure $\cM'$ on $\partial\HH$ associated to $L'$ coincides with $(\phi_A)_*\cM$.
By conformal invariance, it follows that $B(\cM') = \phi_A(B(\cM))$ and so $A$ maps the
conformal structure defined by $B(\cM)$ to the one defined by $B(\cM')$.
Thus, the conformal structures defined on $E^c_q$ through $L$ and through $L'$ coincide.

This completes the explanation of why the probability measure $m_q$ defines a
conformal structure on the vector space $E^c_q$, for each $q \in M\times S$.
Remark~\ref{r.continuity} ensures that this conformal structure varies continuously 
with $q\in M \times  S$.
The fact that the disintegration $\{m_q: q \in M \times  S\}$ is invariant under $f$
and under the holonomies $h^s$ and $h^u$, ensures that the conformal structures obtained
in this way are invariant under the centre derivative $\Dcf$, as well as under the
holonomies $H^s$ and $H^u$.
In this way, every leaf of $\cF^c$ is endowed with a Riemann surface structure,
and these Riemann surface structures are preserved by the dynamical system $f$ and 
its holonomies $h^s$ and $h^u$.

In particular, since $f$ is accessible, the group of conformal automorphisms acts
transitively on every centre leaf $\cF^c_x$. According to \cite[Theorem V.4]{FaK92},
the only compact Riemann surfaces with that property are the sphere $\SS^2$ and the
torus $\TT^2$. This gives alternative (3) in the theorem.

The proof of Theorem~\ref{theorem_general} is now complete.

\subsection{Proof of Corollary~\ref{corollary_general}}

Assume that $f\in\cF^r_\omega(M,S)$ for some symplectic form $\omega$ on $M \times S$.
Then $\dim E^u = \dim E^s$, and we may use the symmetry property \eqref{eq.symmetric}
to conclude that the two centre Lyapunov exponents are symmetric to each other.
Thus, the fact they are distinct, cf. alternative (1) in Theorem~\ref{theorem_general},
implies that they are non-zero, and so $f$ is non-uniformly hyperbolic.

As observed in Section~\ref{s.lyapunov}, the restriction
$\omega^c = \left.\omega\right|_{E^c}$ defines a symplectic form on the
centre leaves which is, clearly, preserved by the restriction of $f$ to the leaves.
Together with the conformal structure, this symplectic form defines an inner product
on each centre space $E^c_q$, varying continuously with $q$.
Thus one gets a continuous Riemannian metric on each centre leaf.

It is clear that this Riemannian metric is invariant under $f$, because both the
conformal structure and the symplectic form are. We have also seen that the conformal
structure is invariant under the stable and unstable holonomies of $f$.
The next lemma asserts that the same is true for the symplectic form
$\omega^c$. So, the Riemannian metric is also invariant under the stable
and unstable holonomies of $f$, as claimed.

This means that we have reduced the proof of Corollary~\ref{corollary_general} to

\begin{lemma}\label{l.symplectic}
For any $x,y\in M$ such that $\cF^c_x$ and $\cF^c_y$ are contained in the same centre-stable leaf and
$p \in \cF^c_x$,
$$
\omega^c_p(u,v) = \omega^c_q\left((Dh^s_{x,y})_p u,(Dh^s_{x,y})_p v \right), 
\qquad q = h^s_{x,y}(p)
$$
for any $u,v \in E^c_p$. A dual statement holds for the unstable holonomy $h^u_{x,y}$
when $x$ and $y$ are such that $\cF^c_x$ and $\cF^c_y$ are contained in the same
centre-unstable leaf.
\end{lemma}

\begin{proof}
Since the centre bundle $E^c$ of $f$ is uniformly close to the vertical direction $\{0\} \times TS$,
horizontal projection induces a linear isomorphism $\Pi_{p.q}:E^c_p \to E^c_q$ between the centre
spaces of any two points $p, q \in M \times S$. The fact that the centre bundle is continuous 
ensures that $\Pi_{p.q}$ is uniformly close to the identity if the distance $\dist(p,q)$ is small.
We denote by $\Delta_{p,q}$ the determinant of $\Pi_{p.q}$ with respect to the symplectic form,
characterised by
\begin{equation}\label{eq.determinant}
\omega^c_q(\Pi_{p,q} u, \Pi_{p,q} v) = \Delta_{p,q} \, \omega_p^c(u,v)
\text{ for any } u, v \in E^c_p.
\end{equation}
Since $\omega^c$ is continuous, we have that $\Delta_{p.q}$ is uniformly close to $1$ if
$\dist(p,q)$ is small.

Given any $x, y \in M$ such that $\cF^c_x$ and $\cF^c_y$ are in the same centre-stable leaf,
and given $p \in \cF^c_x$ and $q = h_{x,y}^s(p)$, we have that (compare \cite[Section~3]{Mar16})
\begin{equation}\label{eq.chain}
\left(Dh^s_{x,y}\right)_p = \lim_n (\Dcf^{-n})_{f^n(q)} \circ \Pi_{f^n(p),f^n(q)} \circ (\Dcf^n)_p
\end{equation}
Thus, recalling that the centre derivative $\Dcf$ preserves the symplectic form $\omega^c$,
$$
\begin{aligned}
\omega^c_q\big((Dh^s_{x,y})_p u,(Dh^s_{x,y})_p v \big)
& = \lim_n \omega^c_q\Big((\Dcf^{-n})_{f^n(q)} \circ \Pi_{f^n(p),f^n(q)} \circ (\Dcf^n)_p u, \\
& \hspace{5cm}(\Dcf^{-n})_{f^n(q)} \circ \Pi_{f^n(p),f^n(q)} \circ (\Dcf^n)_p v \Big) \\
& = \lim_n \omega^c_{f^n(q)} \Big(\Pi_{f^n(p),f^n(q)} \circ (\Dcf^n)_p u,
                                   \Pi_{f^n(p),f^n(q)} \circ (\Dcf^n)_p v \Big) \\
& = \lim_n \Delta_{f^n(p),f^n(q)} \omega^c_{f^n(p)} \Big((\Dcf^n)_p u, (\Dcf^n)_p v \Big)
 = \lim_n \Delta_{f^n(p),f^n(q)} \, \omega^c_{p} \big(u, v\big)
\end{aligned}
$$
for any $u, v \in E^c_p$. Since $\dist(f^n(p),f^n(q))$ converges to zero as $n\to \infty$,
because $p$ and $q$ are in the same stable leaf, the limit on the right hand side
is equal to $\omega^c_p(u,v)$, and so the proof is complete.
\end{proof}

\section{Density of stable accessibility}\label{s.accessibility}\label{s.density_accessibility}

A volume-preserving partially hyperbolic $C^r$, $r\ge 2$ diffeomorphism is said to be
\emph{stably accessible} if every volume-preserving $C^r$ diffeomorphism in a
$C^1$-neighbourhood is accessible. In this section, we check that stable accessibility
is dense among maps of the form  \eqref{eq.affine_extension} and so the accessibility
assumption in Theorem~\ref{theorem_torus} is quite mild.
The proof of the theorem does not depend on this fact.

\begin{theorem}\label{t.accessibility}
Every skew-product \eqref{eq.affine_extension} may be approximated by
another skew-product over $g$ which is stably accessible in the space
of volume-preserving maps.
\end{theorem}

The special case of rotation extensions (that is, $L=\Id$) was proved in \cite{BuW99},
using also ideas from \cite{BP74,Bri75a,Bri75b}, and our arguments here are an adaptation.
We outline how to deal with the presence of the elliptic coefficient $L$, referring the
reader to the previous papers for more details. It is worth mentioning that for circle
rotation extensions in dimension $3$ any perturbation is accessible unless it
(or some finite-order quotient) is smoothly conjugate to the product of an Anosov
diffeomorphism with a rotation~\cite{BPW00}.
Another related result is \cite[Theorem~6.1]{AvV20}.

To better highlight the analogy to the previous papers, in this section we
refer to a slightly more general setting. Namely, we consider maps of the form 
\begin{equation}
f_w:M\times G\to M \times G, \quad
f_w(x,\theta) = (g(x), w(x)L(\theta)),
\end{equation}
where $g$ is a (transitive) Anosov diffeomorphism on a compact manifold,
$G$ is a compact connected Lie group (the case we are most interested in is $G=\TT^2$),
and $L:G\to G$ is an elliptic group isomorphism: by \emph{elliptic} we mean that
$\|DL^n\|$ is bounded uniformly over all $n\in\ZZ$.
We are going to explain that $f_w$ is accessible for an open and dense subset of
$C^r$ maps $w:M\to G$.

functions $w$.

\subsection{Holonomy maps}

Let $\cW^u$ and $\cW^s$ denote, respectively, the unstable foliation and the stable foliation
of the Anosov diffeomorphism $g$. The assumption on $L$ ensures that $f_w$ is partially hyperbolic
and centre-bunched. It is also dynamically coherent:
the leaves of the \emph{centre-unstable foliation} $\cF^{cu}$ are the products
$\cW^u(x) \times G$, $x\in M$ and, analogously, the leaves of the \emph{centre-stable foliation}
$\cF^{cs}$ are the products $\cW^s(x) \times G$, $x\in M$.
Then the \emph{centre foliation} $\cF^c$ coincides with the vertical fibration $\{\{x\}\times G: x \in M\}$.
It also follows that the leaves of the \emph{unstable foliation} $\cF^u$ are
graphs over the unstable manifolds $\cW^u(x)$ of the Anosov diffeomorphism $g$,
and a corresponding fact holds for the \emph{stable foliation} $\cF^s$.

Given points $x$ and $y$ in the same leaf of $\cW^u$, the projection along the
leaves of $\cF^u$ defines a continuous map
$$
h_{x,y}^u: \{x\} \times G \to \{y\} \times G. 
$$
Similarly, the projection along the leaves of $\cF^s$ defines a continuous map
$$
h_{z,w}^s: \{z\} \times G \to \{w\} \times G
$$
for any $z$ and $w$ in the same leaf of $\cW^s$. These are, respectively, unstable and stable holonomies for $f$.

In what follows we sometimes identify a fibre $\{x\} \times G$ to the group $G$ itself,
in the obvious way.
Thus we also view the $h^u_{x,y}$ and $h^s_{w,z}$ as transformations in $G$.

\begin{lemma}\label{l.left-translations}
Every $h^u_{x,y}$ and $h^s_{z,w}$ is given by a left-translation on $G$.
\end{lemma}

\begin{proof}
For $n\in\ZZ$, let us write $f_w^n(x,\theta)=(g^n(x),\cG^n_x(\theta))$.
Then (compare \cite[Proposition~3.2]{ASV13}),
\begin{equation}\label{eq.hu}
h^u_{x,y} = \lim_n \cG^n_{g^{-n}(y)} \circ \cG^{-n}_x.
\end{equation}
An induction argument gives that, for any $n \ge 1$ and $a, b \in M$, 
$$
\begin{aligned}
\cG^n_a
& =
w\left(g^{n-1}(a)\right)
L\left(w\left(g^{n-2}(a)\right)\right)
\cdots 
L^{n-2}\left(w\left(g(a)\right)\right)
L^{n-1}\left(w\left(a\right)\right)
L^n \\ 
\cG^{-n}_b
& = 
L^{-1}\left(w\left(g^{-n}(b)\right)^{-1}\right) L^{-2}\left(w\left(g^{-n+1}(b)\right)^{-1}\right)
\cdots
L^{-n+1}\left(w\left(g^{-2}(b)\right)^{-1}\right)
L^{-n}\left(w\left(g^{-1}(b)\right)^{-1}\right)
L^{-n}.
\end{aligned}
$$
Taking $a = g^{-n}(y)$ and $b = x$, we see that 
$$
\begin{aligned}
\cG^n_{g^{-n}(y)} \circ \cG^{-n}_x(\theta)
& =
w\left(g^{-1}(y)\right)
L\left(w\left(g^{-2}(y)\right)\right)
\cdots 
L^{n-2}\left(w\left(g^{-n+1}(y)\right)\right)
L^{n-1}\left(w\left(g^{-n}(y)\right)\right) \cdot \\
& 
\hspace{1cm} L^{n-1}\left(w\left(g^{-n}(x)\right)^{-1}\right) L^{n-2}\left(w\left(g^{-n+1}(x)\right)^{-1}\right)
\cdots
L^{1}\left(w\left(g^{-2}(b)\right)^{-1}\right)
w\left(g^{-1}(b)\right)^{-1}
\theta
\end{aligned}
$$
is a left-translation for every $n \ge 1$,
and then so is the limit $h^u_{x,y}$.
The argument for $h^s_{z,w}$ is analogous. 
\end{proof}

Next, for each $\tau\in G$, define the \emph{right-translation maps}
$$
\cR_\tau:M\times G \to M\times G, 
\quad
\cR_\tau(x,\theta)=(x,\theta\tau).
$$
Lemma~\ref{l.left-translations} implies that the invariant holonomies
commute with the right-translation maps:
\begin{equation}\label{eq.commute1}
\cR_\tau \circ h^u_{x,y} =  h^u_{x,y} \circ \cR_\tau
\text{ and }
\cR_\tau \circ h^s_{z,w} =  h^s_{z,w} \circ \cR_\tau
\end{equation}
for every $x$, $y$, $z$, $w$ and $\tau$. This means that the unstable
foliation $\cF^u$ and the stable foliation $\cF^s$ are both preserved
by every right-translation map $\cR_\tau$.

\subsection{Holonomy groups}

Let $\gamma$ be any $su$-path in $M$, that is, any piecewise
differentiable oriented curve such that each leg is contained in a leaf of either
$\cW^u$ or $\cW^s$.
Denote by $h_\gamma$ the composition of the unstable holonomies and
the stable holonomies associated to the legs of $\gamma$, in the
natural order. Clearly, $h_\gamma$ is a left-translation and
\begin{equation}\label{eq.commute2}
\cR_\tau \circ h_\gamma = h_\gamma \circ \cR_\tau
\end{equation}
for every $\tau$ and $\gamma$.

Let $e$ denote the neutral element of the group $G$.
For each $x\in M$, define
$$
H_x^0 = \left\{h_\sigma(e): \sigma \text{ is a homotopically null $su$-loop at $x$} \right\}.
$$
Note that $H_x^0$ is a subgroup of $G$.
Indeed, given any $\theta_1=h_{\sigma_1}(e)$ and $\theta_2=h_{\sigma_2}(\theta)$
in $H_x^0$, it follows from \eqref{eq.commute2} that
$$
\theta_1\theta_2
= h_{\sigma_1}(e)\theta_2
= h_{\sigma_1}(\theta_2)
= h_{\sigma_1}\left(h_{\sigma_2}(\theta)\right)
= h_{\sigma_2*\sigma_1}(\theta)
$$
where $\sigma_2*\sigma_1$ is the loop obtained by concatenating $\sigma_2$
with $\sigma_1$ (in this order).
Since the concatenation is still homotopically null, this proves that 
$\theta_1\theta_2\in H_x^0$.
Similarly, given any $\theta=h_\sigma(e)$ denote by $-\sigma$ the loop
obtained by reversing the orientation of $\sigma$. It is clear that
$h_\sigma \circ h_{-\sigma}=\Id$. Since $h_\gamma$ is a left-translation,
this means that
\begin{equation}\label{eq.inverse}
h_\sigma(e)^{-1} = h_{-\sigma}(e).
\end{equation}
It is clear that $-\sigma$ is homotopically null, so this implies that $h_\sigma(e)^{-1}\in H_x^0$.
That completes the proof that $H_x^0$ is a subgroup of $G$.

The \emph{holonomy groups} $H_x^0$ and $H_y^0$ corresponding to two different
points $x$ and $y$ in $M$ are conjugate: there exists $\tau\in G$ such that
\begin{equation}\label{eq.conjugate}
H_y^0 = \tau H_x^0 \tau^{-1}.
\end{equation}
To see this, fix an $su$-path $\beta$ from $x$ to $y$ in $M$ and denote $\tau = h_\beta(e)$.
The relation \eqref{eq.inverse} gives that $h_{-\beta}(e) = \tau^{-1}$.
For any $\theta\in H_x^0$, consider a homotopically null loop $\sigma$ at $x$
such that $\theta=h_\sigma(e)$, and let $\sigma'=\beta*\sigma*(-\beta)$.
Since the holonomy maps are left-translations:
$$
h_{\sigma'}(e)
= h_\beta\left(h_\sigma\left(h_{-\beta}(e)\right)\right)
= h_\beta\left(h_\sigma\left(\tau^{-1}\right)\right)
= h_\beta\left(h_\sigma\left(e\right)\tau^{-1}\right)
= h_\beta\left(\theta\tau^{-1}\right)
= \tau\theta\tau^{-1}.
$$
It is clear that $\sigma'$ is a homotopically null loop at $y$, and so $h_{\sigma'}(e) \in H_y^0$.
Thus this proves that $\tau H_x^0 \tau^{-1}$ is contained in $H_y^0$.
Reversing the roles of $x$ and $y$, we get the other inequality, and thus \eqref{eq.conjugate}
follows.

\subsection{Stable accessibility}

In particular, the following condition is independent of $x\in M$:
\begin{equation}\label{eq.accessible}
H_x^0 = G.
\end{equation}
This is relevant because \eqref{eq.accessible} implies that $f$ is accessible: it clearly implies
that every point in the fibre $\{x\}\times G$ is connected to the unit element $(x,e)$ by an
$su$-path in $M\times G$; for any other fibre $\{y\}\times G$, consider an $su$-path $\beta$
from $x$ to $y$ in $M$ and observe that $(x,e)$ is connected to $h_\gamma(e)\in\{y\}\times G$
by some $su$-path in $M\times G$.

In fact, we have a much stronger fact (see \cite[Theorem~9.1]{BuW99}): \eqref{eq.accessible} implies that
$f$ is \emph{stably accessible}.
The main step is to show that if \eqref{eq.accessible} holds then 
$(x,e)$ is \emph{centrally engulfed from} $(x,e)$: there exist a continuous map 
$\Psi:Z \times [0,1] \to M\times G$ and a constant $N \ge 1$ such that
\begin{itemize}
\item $Z$ is a compact, connected, orientable $\dim E^c$-manifold with boundary;
\item each curve $\Psi(z,\cdot)$ is an $su$-path on $M\times G$ with no more than $N$ legs satisfying
$$
\Psi(z,0)=(x,e)
\text{ and }
\Psi(z,1) \in \{x\} \times G;
$$
\item $\Psi(z,1) \neq (x,e)$ for every $z\in\partial Z$, and the map
$$
\left(Z,\partial Z\right) \to \left(\{x\} \times G, \{x\} \times \left(G\setminus\{e\}\right)\right),
\quad z \mapsto \Psi(z,1)
$$
has positive degree.
\end{itemize}

Then, a degree argument ensures that the accessibility property is stable under
perturbations of the dynamical system.
The arguments in Theorem~9.1 (see also Corollary~5.3) of \cite{BuW99} remain
valid in our setting, so we refer the reader to that paper.
Then, to prove our claim it suffices to check that \eqref{eq.accessible} is a dense property among
the maps of the form \eqref{eq.affine_extension}. This is given by \cite[Proposition~2.3]{Bri75a},
whose proof can be outlined as follows (see \cite[Theorem~9.8]{BuW99}).

In view of the observations, it is no restriction to assume that $x$ is a periodic point.
Recall that the periodic points of a (transitive) Anosov diffeomorphism are dense in the
ambient manifold. Let $m=\dim E^c$ and then choose periodic points $x_1, \dots, x_m$ close
to $x$ and such that they are all in distinct orbits.
The local unstable manifold $W_{loc}^u(x)$ of $x$ intersects the local stable manifold $W^s_{loc}(x_i)$ of $x_i$ at a point $z_i^0$, and the local unstable manifold
$W_{loc}^u(x_i)$ of $x_i$ intersects the local stable manifold $W^s_{loc}(x)$ of $x$ at
a point $z_i^1$. The assumption that $x, x_1, \dots, x_m$ are all in distinct orbits
ensures that the orbits of the $z_i^0$ and $z_i^1$ are all distinct as well.
Fix a neighbourhood $U_i$ of each $z_i^0$, small enough that 
none of the periodic points $x, x_1, \dots, x_m$ is in $U_i$ and, moreover, $f^k(z_j^a)\in U_i$
if and only if $j=i$, $k=0$ and $a=0$.

Let $\sigma_i$ be the $su$-loop consisting of $4$ short legs from $x$ to $z_i^0$, to $x_i$, to $z_i^1$,
and back to $x$. It is clear that $\sigma_i$ is homotopically null.
We denote the associated holonomy map as $h_{\sigma_i,w}$, to highlight the dependence on $w$.
By construction, if $\tilde{w}:M \to G$ coincides with $w$ outside $U_i$ then
$$
h_{\sigma_j,\tilde w}(x,e)=h_{\sigma_j,w}(x,e)
$$
for every $j \neq i$. It is not difficult to find such perturbations $\tilde w$ so that the point
$h_{\sigma_j,\tilde w}(x,e)$ moves in any prescribed direction inside the fibre $\{x\} \times G \approx G$.
Thus, by modifying $w$ suitably inside each $U_1$, \dots, $U_m$, we can ensure that
$$
\left\{h_{\sigma_j,\tilde w}(x,e): i=1, \dots, m\right\}
$$
is not contained in any subgroup of $G$ with dimension less than $m$.
Since this set is contained in the holonomy group $H_{x,\tilde w}^0$,
that implies that $H_{x,\tilde w}^0=G$.

\section{Proof of Theorem~\ref{theorem_torus}}\label{s.theorem_torus}

Let $f_0:M \times \TT^2 \to M \times \TT^2$ be a $C^r$, $r \ge 2$ elliptic affine extension
of a transitive Anosov diffeomorphism $g_0:M\to M$, as defined in \eqref{eq.Mobius_extension}.
Assume that $f_0$ preserves a given measure $\mu$ in the Lebesgue class of $M \times\TT^2$.
It is clear that $f_0$ is a partially hyperbolic skew-product.
In particular (cf. Example~\ref{ex.skew_product}), every $\mu$-preserving
diffeomorphism $f$ in a $C^r$-neighbourhood belongs to $\cF^r_\mu(M,\TT^2)$ and,
thus, satisfies some of the three alternatives in Theorem~\ref{theorem_general}.

The first two alternatives correspond precisely to claims (1) and (2) in the
present Theorem~\ref{theorem_torus}; the second one will be further discussed
in Section~\ref{s.line_fields}. So, we just need to upgrade the alternative (3)
in Theorem~\ref{theorem_general} to the statement in claim (3)
of Theorem~\ref{theorem_torus}.
In what follows we assume that the two centre Lyapunov exponents of
$f:M \times \TT^2 \to M \times \TT^2$ coincide, and the centre leaves are
endowed with Riemann surface structures that vary continuously on $M\times\TT^2$
and are invariant under both the dynamics and the invariant holonomies of $f$.

\subsection{Uniformisation}\label{s.uniformisation}

Given any $\tau\in \mathbb{H}$, let $\TT_\tau^2=\CC/L(1,\tau)$ be the quotient
of the complex plane $\CC$ by the sub-lattice $L(1,\tau)$ generated by $1$
and $\tau$.
Since every centre leaf $\cF^c(q)$ is a topological torus, the corresponding
Riemann surface is a complex torus, and so it admits a Riemann surface
automorphism $\TT_{\tau(q)} \to \cF^c(q)$, for some $\tau(q)\in\HH$.
See \cite[pp. 86--90]{FaK80}.
The Riemann surfaces $\cF^c(q)$ are all conformally equivalent,
as they are mapped to one another by the stable and unstable holonomies,
which preserve the Riemann surface structure.
This implies that the different values of $\tau(p)$ all belong to the same
orbit of the modular group $\PSL(2,\ZZ)$. See~\cite[Section~2]{Hel69}.
The following more precise statement will be useful in what follows:

\begin{proposition}\label{p.uniformisation}
There exist $\tau\in \HH$ and a homeomorphism $\Psi: M \times \TT_\tau^2 \to M \times \TT^2$ whose restriction to each fibre
$\{x\} \times \TT_\tau^2$, $x \in M$ is a Riemann surface automorphism onto
a leaf $\cF^c_x$ of the centre foliation of $f$.
\end{proposition}

\begin{proof}
The first step is to reduce the problem to the case of the unperturbed map $f_0$.
Let $H_f:M\times\TT^2 \to M \times\TT^2$ be a leaf conjugacy as mentioned
in Example~\ref{ex.skew_product}.
Thus $H_f$ is a homeomorphism that  maps each centre leaf $\{x\}\times\TT^2$ of $f_0$
onto a centre leaf $\cF^c_x$ of $f$.
Moreover, each restriction $H^x_f:\{x\}\times\TT^2 \to \cF^c_x$ of the leaf conjugacy
is a $C^1$ diffeomorphism, and the (leaf) derivatives vary continuously on $M \times \TT^2$.
Endow each $\{x\} \times \TT^2$ with the Riemann surface structure that turns $H^x_f$ into a Riemann
surface automorphism. 

These Riemann surfaces are all conformally equivalent and,
also by construction, their conformal structures vary continuously on $M \times \TT^2$.
Now to prove the proposition we only have to find a homeomorphism
$$
\Phi: M \times \TT_\tau^2 \to M \times \TT^2, \quad (x,v) \mapsto (x,\Phi^x(v))
$$
such that each $\Phi^x:\TT_\tau^2 \to \TT^2$ is a Riemann surface automorphism:
then it suffices to take $\Psi = H_f \circ \Phi$.

Let $dz=dx+idy$ and $d\bar{z}=dx-idy$ be the canonical $1$-forms on the torus $\TT^2=\CC/(\ZZ+i\ZZ)$
inherited from the complex plane $\CC$ through the canonical projection $\CC\to\TT^2$.
Then let $\mu^xd\bar{z}/dz$ be the \emph{Beltrami differential} of the Riemann
surface structure  on $\{x\}\times\TT^2$ discussed in the previous paragraphs.
In other words, $\mu^x:\TT^2 \to \DD$ is such that the metric $ds = |dz+\mu^x(v) d\bar{z}|$
belongs to the conformal structure at each point $(x,v)\in\{x\}\times\TT^2$.
The fact that these conformal structures vary continuously on $M \times \TT^2$
ensures that
$$
\mu: M \times \TT^2 \to \DD, \quad (x,v) \mapsto \mu^x(v)
$$
is a continuous function. By compactness, it follows that $k=\sup|\mu|$
is strictly smaller than $1$.

Via the canonical projection $\CC\to\TT^2$, we may view each $\mu^x$ as
a continuous $\ZZ+i\ZZ$-periodic function on the complex plane, with
$$
\sup|\mu^x| \le k < 1 \text{ for every } x \in M.
$$
Then, by the measurable Riemann mapping theorem (Theorem~\ref{t.mRmt}),
there exists a unique homeomorphism $w^x:\bar\CC\to\bar\CC$ that fixes
$0$, $1$, and $\infty$ and satisfies the \emph{Beltrami equation}
\begin{equation}\label{eq.Beltrami2}
\partial_{\bar z} w^x = \mu^x \partial_z w^x,
\end{equation}
which means that $w^x$ maps the conformal structure defined by $\mu^x$ to the standard conformal
structure on $\CC$. Any other solution of \eqref{eq.Beltrami2} is obtained from $w^x$ through
post-composition with a Möbius automorphism of the complex plane.
Moreover, using Theorem~\ref{t.mRmt_continuity}, the homeomorphism $w^x$
depends continuously on the function $\mu^x$ in the sense that for any
compact set $K\subset\CC$ there exists $C(K)>0$ such that 
\begin{equation}\label{eq.continuous}
\sup_{z\in K}|w^x(z)-w^y(z)| \le C(K) \sup |\mu^x-\mu^y|
\text{ for any $x, y \in M$.}
\end{equation}
Consequently, $w^x$ depends continuously on $x \in M$, uniformly on each compact subset of $\CC$.

We claim that 
\begin{equation}\label{eq.commutations}
w^x(z+1) = w^x(z)+1 \text{ and } w^x(z+i) = w^x(z)+\tau(x) \text{ for every } z \in \CC,
\end{equation}
where $\tau(x) = w^x(i)$. Keep in mind that $w^x(1)=1$.
Indeed, since $\mu^x$ is $\ZZ+i\ZZ$-periodic, both $z \mapsto w^x(z+1)$ and 
$z \mapsto w^x(z+i)$ are solutions of the Beltrami equation that fix $\infty$. 
Thus, there exist Möbius automorphisms $M_1(z) = a_1 z + b_1$ and 
$M_i(z) = a_i z + b_i$ such that 
$$
w^x(z+1) = M_1 \circ w^x(z) \text{ and } w^x(z+i) = M_i \circ w^x(z)
$$
for every $z$. It is clear that $b_1=w^x(1)=1$ and $b_i=w^x(i)=\tau(x)$.
If $a_1 \neq 1$ then $M_1$ has a finite fixed point $p_1=b_1/(1-a_1)$. 
Let $z_1 = (w^x)^{-1}(p_1)$. Then 
$$
w^x(z_1 + n) = M_1^n(w^x(z_1)) = p_1 \text{ for every } n.
$$
This contradicts the fact that $w^x$ is injective, and that contradiction
proves that $a_1=1$. The same argument proves that $a_i=1 $ and so the
claim \eqref{eq.commutations} is proved.

Now, \eqref{eq.commutations} ensures that $w^x(\ZZ+i\ZZ) = L(1,\tau(x))$,
and so $w^x$ descends to a homeomorphism 
$$
W^x: \TT^2 \to \TT_{\tau(x)}.
$$
Since the conformal structures defined by the $\mu^x$ are all conformally equivalent,
as pointed out before, we have that the different values of $\tau(x)$ all belong to
the same orbit of the modular group $\SL(2,\ZZ)$. See \cite[Section~2]{Hel69}.
It is also clear from the definition $\tau(x) = w^x(i)$ that $\tau(x)$ depends continuously
on $x$. Since the modular group is discrete, it follows that the function $x \mapsto \tau(x)$ is constant. Denote by $\tau$ that constant.

Now just take $\Phi^x:\TT_\tau^2 \to \TT^2$ to be the inverse of $W^x$, for every $x \in M$.
\end{proof}

\subsection{Translation structures}\label{s.translation_structures}

The Riemann surface $\TT_\tau^2$ also carries a canonical structure of a translation surface,
inherited from the complex plane $\CC$. In what follows we consider on each centre leaf $\cF^c_x$ the
translation structure transported from $\{x\} \times \TT_\tau^2$ through the uniformisation map $\Psi$.

\begin{lemma}~\label{holonomytranslation}
The stable and unstable holonomies of $f$ are translations with respect to the translation structures
on the centre leaves.
\end{lemma}

\begin{proof}
For any $p$ and $q$ in the same stable leaf of $f$, let $h^s_{p,q}:\cF^c(p) \to \cF^c(q)$ be
the stable holonomy and
$$
\mh^s_{p,q}: \TT_\tau^2 \to \TT_\tau^2, \quad \mh^s_{p,q} = \Psi_q^{-1} \circ h^s_{p,q} \circ \Psi_p
$$
be its expression under the uniformisation provided by Proposition~\ref{p.uniformisation}.
This is a conformal homeomorphism of $\TT_\tau^2$ and so it lifts to a conformal automorphism $\hat\mh^s_{p,q}:\CC\to\CC$,
that is, a map of the form
$$
\hat\mh^s_{p,q}(z)=\alpha_{p,q}z+\beta_{p,q}
$$
for some complex numbers $\alpha_{p,q}$ and $\beta_{p,q}$.  
Since the uniformisation map $\Psi$ is a homeomorphism, $\mh^s_{p,q}(z)$ is a continuous function of
$p$, $q$ and $z$. That implies that $\alpha_{p,q}$ depends continuously on $p$ and $q$.
According to \cite[Theorem~7]{Hel69}, the set of all possible values of $\alpha_{p,q}$ is discrete.
This implies that $\alpha_{p,q}$ is actually independent of $p$ and $q$.
On the other hand, as $d(p,q) \to 0$ the stable holonomy map $\mh^s_{p,q}$ converges to
the identity, and then so does its lift $\hat\mh_{p,q}^s$.
This means that $\alpha_{p,q}\to 1$ as $d(p,q)\to 0$, and so $\alpha_{p,q}=1$ for every
$p$ and $q$. This proves that the stable holonomy map is a translation
for any $p$ and $q$ in the same stable leaf.

The same argument applies to the unstable holonomies.
\end{proof} 

Let $g$ be the map induced by $f$ on the space of centre leaves, which we may view as the homeomorphism
$g: M \to M$ defined by
\begin{equation}\label{eq.defg}
f(\cF^c_x) = \cF^c_{g(x)}.
\end{equation}
Since the leaf conjugacy $H_f:M \times \TT^2 \to M \times \TT^2$ maps centre leaves
of $f_0$ to centre leaves of $f$, it descends to a homeomorphism $h_f:M\to M$.
The invariance property \eqref{eq.leaf_conjugacy} means that this $h_f$
conjugates $g$ to $g_0$:
\begin{equation}\label{eq.conjugahyp}
g \circ h_f = h_f \circ g_0.
\end{equation}

Recall that the uniformisation $\Psi$ in Proposition~\ref{p.uniformisation} maps each fibre $\{x\}\times \TT_\tau^2$ to the
centre leaf $\cF^c_x$. Thus, 
$$
\Psi^{-1} \circ f \circ \Psi: M \times \TT_\tau^2 \to M \times \TT_\tau^2,
\quad
(x,z) \mapsto (g(x), \mf_x(z))
$$
where $\mf_x:\TT_\tau^2 \to \TT_\tau^2$ is given by
$$
\mf_x = (\Psi \mid_{\{g(x)\} \times \TT_\tau^2})^{-1} \circ f \circ \left( \Psi \mid_{\{x\} \times \TT_\tau^2}\right).
$$
This is an invertible conformal map, and so its lift $\phi_x:\CC\to\CC$ is a conformal automorphism of the plane.
It follows that $\phi_x(z) = a_xz+b_x$ for some $a_x, b_x \in \CC$.
In particular, the Jacobian of $\phi_x$ relative to the standard area form on $\CC$ is constant equal to $|a_x|^2$,
and then so is the Jacobian of $\mf_x$ relative to the standard area form on $\TT_\tau^2$
(inherited from the complex plane). Since the torus has finite area, this implies that $|a_x|=1$.

Next, consider the homeomorphism $\Lambda = h_f \times \lambda:M \times \TT^2 \to M \times \TT_\tau^2$
where $\lambda:\TT^2 \to \TT_\tau^2$ is the map induced by the $\RR$-linear isomorphism $\CC\to\CC$
that fixes $1$ and sends $i$ to $\tau$ (and which we also denote as $\lambda$). Then, denoting
$Y = \Psi \circ \Lambda$,
\begin{equation}\label{eq.conjug1}
Y^{-1} \circ f \circ Y:M\times \TT^2 \to M \times \TT^2,
\quad 
(x,v) \mapsto (g_0(x),\lambda^{-1}(a_x\lambda(v)+b_x)).
\end{equation}
The fact that the affine map $\CC\to\CC$, $v \mapsto \lambda^{-1}(a_x\lambda(v)+b_x)$ descends to a homeomorphism
of the torus means that the linear part
$$
L_x:\CC \to \CC, \quad u \mapsto \lambda^{-1}(a_x\lambda(u))
$$
preserves the lattice $\ZZ^2$, and that means that $L_x \in \SL(2,\ZZ)$. 
Since the latter is a discrete group, and the map $x \mapsto L_x$ is continuous, 
it follows that $x \mapsto L_x$ is actually constant, and then so is $x \mapsto a_x$.
The spectrum of $L=L_x$ consists of $a=a_x$ and its conjugate, and so $L$ is elliptic.
By Remark~\ref{r.finite_order}, it follows that $L^n = \Id$ or, equivalently, $a^n = 1$
for some $n \in \{1, 2, 3, 4, 6\}$. Finally, we may write \eqref{eq.conjug1} as
\begin{equation}\label{eq.conjug2}
Y^{-1} \circ f \circ Y:M\times \TT^2 \to M \times \TT^2,
\quad 
(x,v) \mapsto (g_0(x),L v + w'(x))
\end{equation}
for some continuous function $w':M\to \TT^2$.
Now, to complete the proof of claim (3) in the theorem we only need
to explain why $L$ may be taken equal to $L_0$.

On the one hand, it is clear that the restriction 
\begin{equation}\label{eq.f0}
\{x\} \times \TT^2 \to \{g_0(x)\} \times \TT^2, \quad 
(x,v) \mapsto (g_0(x), L_0v + w_0(x))
\end{equation}
 of $f_0$ to each centre leaf is isotopic to the linear automorphism $L_0 : \TT^2 \to \TT^2$.
The leaves $\cF^c_x$ of the centre foliation of $f$ are uniformly close to the vertical fibres
$\{x\} \times \TT^2$, and so each may be identified with $\TT^2$ via the horizontal projection
$(x,v) \mapsto v$. In this way, every restriction
$$
\left(f \mid \cF^c_x\right):\cF^c_x \to \cF^c_{g(x)}
$$
may be viewed as a map $f_x: \TT^2 \to \TT^2$. By construction, these maps are uniformly close
to \eqref{eq.f0}, and so they are all isotopic to $L_0:\TT^2 \to \TT^2$. In particular,
the action of every $f_x$ on the homology of the torus is given by $L_0$.

On the other hand, \eqref{eq.conjug2} gives that $f_x$ is topologically
conjugate to a map of the form $v \mapsto L v + w(x)$,
and so their actions on the homology of the torus are linearly conjugate. 
The latter map is isotopic to $L: \TT^2 \to \TT^2$, and so its action on the homology of the torus
is given by $L$. This shows that $L$ and $L_0$ are linearly conjugate, that is, there exists
$P \in \SL(2,\ZZ)$ such that $L = P^{-1} L_0 P$. Then, denoting $Z=(\Id \times P) \circ Y$,
\begin{equation}\label{eq.conjug3}
Z^{-1} \circ f \circ Z
: M\times \TT^2 \to M \times \TT^2,
\quad 
(x,v) \mapsto (g_0(x), L_0 v + w(x)),
\end{equation}
with $w(x) = P^{-1} w'(x)$.
This finishes the proof of Theorem~\ref{theorem_torus}.

\section{Invariant line fields}\label{s.line_fields}

We begin by proving Corollary~\ref{corollaryD}. Then we present a simple
example where invariant line fields as in alternative (2) in
Theorem~\ref{theorem_torus} do occur.

Let $x_0 \in M$ be a fixed point of $g_0$. The transformation
$\TT^2 \to \TT^2,\  v \mapsto L_0(v) + w_0(x_0)$ lifts to a map
$$
\RR^2 \to \RR^2, \quad v \mapsto L_0(v) + W_0
$$
where $W_0\in\RR^2$ is any vector that projects to $w_0(x_0)$ under the
covering map $\RR^2\to\TT^2$.
Our assumptions ensure that $1$ is not an eigenvalue of $L_0$, that is,
the linear map $\Id-L$ is invertible.
Let $v_0$ be the projection to the torus $\TT^2$ of the 
vector $(\Id-L_0)^{-1}(W_0)$. Then $p_0=(x_0,v_0)$ is a fixed point of $f_0$,
and it is easy to see that this fixed point is \emph{simple}: the spectrum of
$$
Df_0(p_0)= \left(\begin{array}{cc}
Dg_0(x_0) & 0\\
Dw_0(x_0) & L_0
\end{array}\right)
$$
is the union of the spectra of $Dg_0(x_0)$ and $L_0$, and thus does not contain $1$. 
Consequently, every diffeomorphism $f$ in a neighbourhood of $f_0$ has a unique fixed point $p$ close to $p_0$,
and this fixed point is still simple.
We refer to $p$ as the \emph{continuation} of the fixed point $p_0$ of $f$. 

Let $f \in \cU_T$ be an accessible $\mu$-preserving diffeomorphism, and suppose that its centre Lyapunov exponents coincide.
Let $m$ be any $\PP(\Dcf)$-invariant measure $m$ on $\PP(E^c)$ projecting down to $\mu$ on $M\times\TT^2$.
By the Invariance Principle (Theorem~\ref{t.invariance_principle}), $m$ 
admits a continuous disintegration $\{m_q: q \in M \times \TT^2\}$
invariant under the dynamics and under the stable and unstable holonomies,
that is, satisfying \eqref{eq.su-invariance} and \eqref{eq.invariance}.

\begin{lemma}\label{l.noatoms}
Assuming $f$ is close enough to $f_0$, the conditional probabilities $m_q$
can have no atoms of mass greater than or equal to $1/2$.
\end{lemma}

\begin{proof}
Clearly, the number of atoms of each $m_q$ with mass greater than or equal to $1/2$ is at most $2$.
Moreover, that number does not depend on $q$, because the disintegration is holonomy invariant and $f$ is
assumed to be accessible.
Let $p$ be the continuation of the fixed point $p_0$ for $f$.
Since,
by \eqref{eq.invariance}, the probability measure $m_p$ is invariant under $\PP((\Dcf)_p)$,
any atoms with mass greater than or equal to $1/2$ must be periodic points of period $1$ or $2$.
However, such periodic points cannot exist if $f$ is close to $f_0$,
because then $\PP((\Dcf)_p)$ is close to $\PP((\Dcf_0)_{p_0})=\PP(L_0)$
which, by assumption, has no periodic points with period less than $3$.
\end{proof} 

This means that the alternative (2) in Theorem~\ref{theorem_torus} cannot
occur in the present setting. Thus $f$ must be satisfy alternative (3),
that is, it must be topologically conjugate to an $L_0$-affine extension
of $g_0$. This proves Corollary~\ref{corollaryD}.

On the other hand, the alternative (2) may occur for volume-preserving,
and even symplectic,  diffeomorphisms arbitrarily close to $f_0$ when $L_0=\Id$ :

\begin{example}\label{ex.linefield}
Consider $f_0:M\times\TT^2 \to M\times\TT^2$ given by $f_0(x,v)=(g_0(x),v)$.
Assume that $g$ preserves some symplectic form $\omega_M$ on $M$ and let $\omega_S$
be the standard area form on $\TT^2$.
Then $f_0$ preserves the symplectic form $\omega=\omega_M \times \omega_S$
on $M\times\TT^2$. By Theorem~\ref{t.accessibility}, for any $r\ge 2$ there
exist functions $w:M\to\TT^2$ arbitrarily $C^r$-close to zero such that
the diffeomorphism
$$
f_w:M\times\TT^2 \to M \times\TT^2, \quad f_w(x,v) = (g(x),v+w(x))
$$
is stably accessible. Observe that $f_w$ is $\omega$-symplectic.
Next, define 
$$
f:M\times\TT^2 \to M \times \TT^2,
\quad f\left(x,(v_1,v_2)\right) = \left(g(x),w(x)+(v_1+\alpha(v_2),v_2)\right),
$$
where $\alpha:\SS^1\to\RR$ is a smooth function. Assuming that $\alpha$
is $C^r$-close to zero, $f$ is $C^r$ close to $f_w$ and, hence, it is accessible.
It is clear that $f$ can be made arbitrarily close to $f_0$ by picking both $\alpha$ and $w$
sufficiently small. In particular, $f$ is partially hyperbolic, centre-bunched and
dynamically coherent. Moreover, $f$ is itself a skew-product,
and the vertical fibration $\{x\} \times \TT^2$ is the centre foliation.
It is also clear that $f$ is $\omega$-symplectic.
By construction, the centre derivative of $f$ is idempotent:
\begin{equation}\label{eq.idempotent}
\Dcf_{x,v} = \left(\begin{array}{cc} 1 & \alpha'(v_2)\\ 0 & 1\end{array} \right).
\end{equation}
In particular, the horizontal line bundle $H_{x,v} \equiv (1,0)$ is invariant under $\Dcf$.
Observe that $m = \volume \times \delta_{H_{x,v}}$ is a $\PP(\Dcf)$-invariant probability
measure on $\PP(E^c)$ that projects down to the volume measure on $M\times\TT^2$.
\end{example}

\section{Proof of Theorem~\ref{theorem_sphere}}\label{s.theorem_sphere}

Take $f_0:M \times \SS^2 \to M \times \SS^2$ to be a $C^r$, $r\ge 2$ Möbius extension of a
transitive Anosov diffeomorphism $g_0:M\to M$, as defined in \eqref{eq.Mobius_extension}.
Assume that $f_0$ preserves a given measure $\mu$ in the Lebesgue class of $M \times\SS^2$.
It is clear that $f_0$ is a partially hyperbolic skew-product.
In particular (cf. Example~\ref{ex.skew_product}), every $\mu$-preserving
diffeomorphism in a $C^r$-neighbourhood belongs to $\cF^r_\mu(M,\SS^2)$
and, thus, satisfies some of the three alternatives
in Theorem~\ref{theorem_general}.

The first alternative corresponds precisely to the claim (1) in the present Theorem~\ref{theorem_sphere}.
The second one is excluded here, because the sphere $\SS^2$ supports neither continuous line fields nor
continuous pairs of line fields. In the third one, the centre leaves are endowed with continuous Riemann
surface structures invariant under $f$ and under the stable and unstable holonomies.
We are left to checking that this yields the claim (2) of Theorem~\ref{theorem_sphere}.

\begin{proposition}\label{p.uniformisation2}
There exists a homeomorphism $\Psi: M \times \SS^2 \to M \times \SS^2$ whose restriction to each
fibre $\{x\} \times \SS^2$, $x \in M$ is a Riemann surface automorphism onto the centre leaf
$\cF^c_x$ of $f$.
\end{proposition}

\begin{proof}
Let $H_f:M\times\SS^2 \to M \times\SS^2$ be a leaf conjugacy as mentioned in
Example~\ref{ex.skew_product}.
Thus $H_f$ is a homeomorphism that  maps each centre leaf $\{x\}\times\SS^2$ of $f_0$ onto
a centre leaf $\cF^c_x$ of $f$.
Moreover, each restriction $H^x_f:\{x\}\times\SS^2 \to \cF^c_x$ of the leaf conjugacy is a $C^1$
diffeomorphism, and the leaf derivatives vary continuously on $M \times \SS^2$.
Endow $\{x\} \times \SS^2$ with the Riemann surface structure that turns $H^x_f$ into a Riemann
surface automorphism. This structure may be represented by a function $\mu^x:\SS^2 \to \DD$,
such that the metric $ds = |dz+\mu^x(v) d\bar{z}|$ belongs to the conformal structure at each
point $(x,v)\in\{x\}\times\SS^2$. The fact that these conformal structures vary continuously
on $M \times \SS^2$ ensures that the function
$$
\mu: M \times \SS^2 \to \DD, \quad (x,v) \mapsto \mu^x(v)
$$
is continuous.

By the measurable Riemannian mapping theorem (Theorem~\ref{t.mRmt}),
for each $x \in M$ there is a unique homeomorphism
$w^x:\bar\SS^2\to\bar\SS^2$ that fixes $0$, $1$, and $\infty$ and
satisfies the Beltrami equation
\begin{equation}\label{eq.Beltrami3}
\partial_{\bar z} w^x = \mu^x \partial_z w^x,
\end{equation}
which means that $w^x$ maps the conformal structure defined on $\SS^2$ by $\mu^x$ to the standard
conformal structure on $\SS^2$. Any other solution of \eqref{eq.Beltrami3} is obtained from $w^x$
through post-composition with a Möbius automorphism of the sphere.
Moreover, by Theorem~\ref{t.mRmt_continuity}, $w^x$ depends continuously on
the function $\mu^x$, uniformly on $\SS^2$.
Consequently, $w^x$ depends continuously on $x \in M$, uniformly on the sphere, and so the map
$$
W:M \times \SS^2 \to M \times \SS^2, \quad (x,v) \mapsto (x,w^x(v))
$$
is a homeomorphism.
To complete the proof, define $\Psi = H_f \circ W^{-1}:M\times \SS^2 \to M \times \SS^2$.
By construction, $\Psi$ maps each $\{x\}\times\SS^2$ conformally to the corresponding
centre leaf $\cF^c_x$.
\end{proof}

It follows from Proposition~\ref{p.uniformisation2} that we may write
$$
\Psi^{-1} \circ f \circ \Psi:M\times\SS^2 \to M \times \SS^2,
\quad (x,v) \mapsto (g(x),\tilde\zeta_x(v))
$$
where $g:M\to M$ is the Anosov homemorphism in \eqref{eq.defg}, and each $\tilde\zeta_x:\SS^2\to\SS^2$
is a conformal automorphism of the sphere, that is, a Möbius transformation.
Let $h_f:M\to M$ be the conjugacy between $g$ and $g_0$, as in \eqref{eq.conjugahyp}. Then
$$
(h_f \times\Id)^{-1} \circ \Psi^{-1} \circ f \circ \Psi \circ (h_f\times\Id):M\times\SS^2 \to M \times \SS^2,
\quad (x,v) \mapsto (g_0(x),\zeta_x(v))
$$
with $\zeta_x=\tilde\zeta_{h_f(x)}$. This completes the proof of Theorem~\ref{theorem_sphere}.

\bibliographystyle{alpha}  
\bibliography{bib}  

\end{document}